\documentclass[11pt]{amsart}

\usepackage[a4paper, margin=3cm]{geometry}
\usepackage{amsmath,amssymb}
\allowdisplaybreaks
\usepackage[colorlinks=true, pdfstartview=FitV, linkcolor=blue, citecolor=blue, urlcolor=blue]{hyperref}

\newtheorem{definition}{Definition}[section]
\newtheorem{theorem}[definition]{Theorem}
\newtheorem{proposition}[definition]{Proposition}
\newtheorem{corollary}[definition]{Corollary}
\newtheorem{lemma}[definition]{Lemma}
\newtheorem{priorresult}{Theorem}

\theoremstyle{remark}
\newtheorem{remark}[definition]{Remark}

\numberwithin{equation}{section}

\newcommand{\pv}{\operatorname{p.v.}}
\newcommand{\R}{\mathbb R}
\newcommand{\C}{\mathbb C}
\newcommand{\N}{\mathbb N}
\newcommand{\T}{\mathbb T}
\newcommand{\Sph}{\mathbb S^1}
\newcommand{\cK}{\mathcal K}
\newcommand{\cM}{\mathcal M}
\newcommand{\cI}{\mathcal I}
\newcommand{\cU}{\mathcal U}
\newcommand{\norm}[1]{\left\lVert #1\right\rVert}
\newcommand{\abs}[1]{\left\lvert #1\right\rvert}
\newcommand{\sgn}{\operatorname{sgn}}

\newcommand{\one}{\mathbf 1}

\newcommand{\Var}{\operatorname{Var}}

\newcommand{\supp}{\operatorname{supp}}

\newcommand{\p}{\partial}

\begin{document}
	
	\author[B. Dan]{Binwei Dan}
	\address{Binwei Dan:
		School of Mathematical Sciences \\
		Beijing Normal University \\
		Laboratory of Mathematics and Complex Systems \\
		Ministry of Education \\
		Beijing 100875 \\
		People's Republic of China}
	\email{danbinwei@mail.bnu.edu.cn}
	
	\author[Q. Xue]{Qingying Xue\(^{*}\)}
	\address{Qingying Xue:
		School of Mathematical Sciences \\
		Beijing Normal University \\
		Laboratory of Mathematics and Complex Systems \\
		Ministry of Education \\
		Beijing 100875 \\
		People's Republic of China}
	\email{qyxue@bnu.edu.cn}
	
	\keywords{Rough bilinear singular integrals; angular multipliers; \(L\log L\) kernel;
		directional kernel conditions; Grafakos--Stefanov condition.}
	\subjclass[2020]{Primary 42B20; Secondary 42B35, 42B25.}
	
	\thanks{The authors were partly supported by the National Key R\&D Program of China (No. 2020YFA0712900) and NNSF of China (No. 12271041).}
	\thanks{\(^{*}\) Corresponding author, e-mail address: qyxue@bnu.edu.cn.}
	
	\date{\today}
	\title[Endpoint criteria for bilinear rough singular integrals]
	{Endpoint Criteria for One-Dimensional Bilinear Rough Singular Integrals}
	
	\begin{abstract}
	We prove endpoint theorems for  one-dimensional bilinear rough singular integrals.  The point of departure is a sharp structural description of the
	angular multiplier.  For every mean-zero \(\Omega\in L^1(\Sph)\), the
	finite-part angular multiplier associated with \(T_\Omega\) has bounded
	variation if and only if the antipodal even part of \(\Omega\) belongs to
	\(H^1(\Sph)\).  This identifies the exact rotational regularity needed for
	the one-dimensional bilinear problem and gives a Stieltjes decomposition
	compatible with uniform bilinear Hilbert transform estimates.
	We then deduce two endpoint boundedness criteria.  First, if
	\(\Omega\in L\log L(\Sph)\), then
	$
	T_\Omega:L^{p_1}(\R)\times L^{p_2}(\R)\to L^p(\R)
	$
	for all \(1<p_1,p_2,p<\infty\) with
	\(1/p=1/p_1+1/p_2\).  The logarithmic exponent \(1\) is optimal in the
	scale \(L(\log L)^A\).  Second, at the critical directional threshold,
	if \(\Omega\in\cK_{1/2,\beta}(\Sph)\), then the same boundedness holds
	whenever
	$
	\beta>\frac32\max\{p_1,p_1',p_2,p_2'\}-1 .
	$
	These two endpoint kernel classes are  incomparable.  The Orlicz
	endpoint is obtained by reducing the multiplier to a finite-part angular
	profile of bounded variation; the directional endpoint is proved through
	endpoint Fourier decay, product wavelet decompositions, and interpolation.
	\end{abstract}
	\maketitle
	\tableofcontents
	
	\section{Introduction and main results}
	
		\subsection{Background and motivation}This paper concerns the endpoint theory of the one-dimensional bilinear rough
	singular integral
	\begin{equation}\label{eq:T-def}
		T_\Omega(f,g)(x)
		:=
		\pv\int_{\R^2}
		\frac{\Omega(y/\abs y)}{\abs y^2}
		f(x-y_1)g(x-y_2)\,dy_1\,dy_2,
	\end{equation}
	initially defined for \(f,g\in\mathcal S(\R)\).  Here \(\Omega\) is a
	mean-zero function on the unit circle \(\Sph\), and no smoothness is imposed on
	\(\Omega\).   We normalize surface
	measure by \(\sigma(\Sph)=1\) and assume
	\(\int_{\Sph}\Omega\,d\sigma=0\).  Unless stated otherwise, all norms and
	integrals over \(\Sph\) are taken with respect to this normalized measure.  For
	\(1<r<\infty\), \(r'\) denotes the H\"older conjugate exponent.
	
	The problem is to determine how little angular regularity is needed for
	\eqref{eq:T-def} to be bounded.  In the linear theory this question goes back
	to Calder\'on and Zygmund and is closely tied to near-\(L^1\) assumptions such
	as \(L\log L\); see
	\cite{CalderonZygmund1956,DuoandikoetxeaRubio1986,Christ1988,Seeger1996,
		GrafakosStefanov1998,Tao1999}.  For bilinear and multilinear rough singular
	integrals, substantial progress has been made under \(L^q\)-type hypotheses,
	sparse and weighted estimates, and endpoint criteria; see
	\cite{CoifmanMeyer1975,GrafakosHeHonzik2018,GrafakosHeSlavikova2020,
		Barron2017,GrafakosWangXue2022,HePark2023,
		GrafakosHeHonzikPark2023}.
	
	We recall the results closest to the endpoint considered here.  They are stated
	in their original \(n\)-dimensional form.  For \(n\ge1\) let
	\(d\sigma_n\) denote normalized surface measure on \(\mathbb S^{2n-1}\), and set
	\[
	T_\Omega^{(n)}(f,g)(x)
	:=
	\pv\int_{\R^n\times\R^n}
	\frac{\Omega\bigl((y,z)/|(y,z)|\bigr)}{|(y,z)|^{2n}}
	f(x-y)g(x-z)\,dy\,dz,
	\qquad x\in\R^n.
	\]
	Thus \(T_\Omega^{(1)}\) is the operator in \eqref{eq:T-def}.  The superscript is
	used only in this discussion of prior results.
	He and Park proved the following improved \(L^q\) estimate.
	\begin{priorresult}[\cite{HePark2023}]
		\label{thm:prior-A}
		Let \(n\ge1\),
		$
		1<p_1,p_2\le\infty,
		\,
		\frac12<p<\infty,
		\,
		\frac1p=\frac1{p_1}+\frac1{p_2}.$
		Suppose that
		$\max\left\{\frac43,\frac{p}{2p-1}\right\}<q\le\infty$
		and that the kernel
		\(\Omega\in L^q(\mathbb S^{2n-1})\) has mean zero. 
		Then
		\[	
		\norm{T_\Omega^{(n)}(f,g)}_{L^p(\R^n)}
		\le
		C_{n,p_1,p_2,q}\norm{\Omega}_{L^q(\mathbb S^{2n-1})}
		\norm f_{L^{p_1}(\R^n)}\norm g_{L^{p_2}(\R^n)}.
		\]
	\end{priorresult}
	
	For finite input exponents, Dosidis and Slav\'ikova later removed the lower
	bound \(q>4/3\).  Their result gives boundedness for every \(q>1\), subject to
	the condition \(1/p+1/q<2\).
	
	\begin{priorresult}[\cite{dosidis_multilinear_2024}]
		\label{thm:prior-B}
		Let \(n\ge1\), \(q>1\), and $
		1<p_1,p_2<\infty,
		\,
		\frac1p=\frac1{p_1}+\frac1{p_2}.
		$
		Suppose that \(1/p+1/q<2\) and that the kernel
		\(\Omega\in L^q(\mathbb S^{2n-1})\) has mean zero.  Then
		\[
		\norm{T_\Omega^{(n)}(f,g)}_{L^p(\R^n)}
		\le
		C_{n,p_1,p_2,q}\norm{\Omega}_{L^q(\mathbb S^{2n-1})}
		\norm f_{L^{p_1}(\R^n)}\norm g_{L^{p_2}(\R^n)}.
		\]
	\end{priorresult}
	
	Theorems~\ref{thm:prior-A} and~\ref{thm:prior-B} still require \(q>1\).
	Dosidis, Park, and Slav\'ikova obtained a near-\(L^1\) substitute in the
	Orlicz scale \cite[Theorem~3]{DosidisParkSlavikova2026}.  For \(A\ge0\), let
	\(L(\log L)^A\) denote the Orlicz space associated with
	\(\Phi_A(t)=t\log^A(e+t)\).
	
	\begin{priorresult}[\cite{DosidisParkSlavikova2026}]
		\label{thm:prior-C}
		Let \(n\ge1\),
		$
		1\le p<\infty,
		\,
		1<p_1,p_2\le\infty,
		\,
		\frac1p=\frac1{p_1}+\frac1{p_2},
		$
		and let \(p'\) be the conjugate exponent of \(p\), with \(p'=\infty\)
		when \(p=1\).  Suppose that
		$
		A\ge1+\max\left\{
		\frac1{p_1},\frac1{p_2},\frac1{p'}
		\right\}
		$
		and that the mean-zero kernel satisfies
		\(\Omega\in L(\log L)^A(\mathbb S^{2n-1})\).  Then, for all
		\(f,g\in\mathcal S(\R^n)\),
		\[
		\norm{T_\Omega^{(n)}(f,g)}_{L^p(\R^n)}
		\le
		C_{n,p_1,p_2}
		\norm{\Omega}_{L(\log L)^A(\mathbb S^{2n-1})}
		\norm f_{L^{p_1}(\R^n)}\norm g_{L^{p_2}(\R^n)}.
		\]
	\end{priorresult}
	
	In the case \(n=1\), Bhojak and Shrivastava
	\cite[Theorems~1.1 and~1.3]{BhojakShrivastava2025} gave an alternative
	local-Fourier-series proof of the \(L^q\) bounds in
	Theorem~\ref{thm:prior-B}.  For finite exponents, their method also gives the
	corresponding Orlicz bound with logarithmic exponent
	$
	A=1+\max\left\{
	\frac1{p_1},\frac1{p_2},\frac1{p'}
	\right\}.
	$
	This still leaves the \(L\log L\) endpoint outside the known theory.
	
	We now return to \(n=1\) and write \(T_\Omega\) for \(T_\Omega^{(1)}\).
	Hereafter, the finite-exponent Banach range means
	\[
	1<p_1,p_2,p<\infty,
	\qquad
	\frac1p=\frac1{p_1}+\frac1{p_2}.
	\]
	The preceding results approach the endpoint from above, but they do not reveal
	the one-dimensional endpoint structure.  Our motivation comes from the
	following three considerations.
	\begin{itemize}
		\item Existing \(L^q\) estimates require \(q>1\), whereas the near-\(L^1\)
		estimate in Theorem~\ref{thm:prior-C} assumes
		\(\Omega\in L(\log L)^A(\Sph)\) with
		\[
		A\ge 1+\max\left\{
		\frac1{p_1},\frac1{p_2},\frac1{p'}
		\right\}.
		\]
		It is therefore natural to ask whether \(\Omega\in L\log L(\Sph)\) suffices
		on the line, and whether the exponent \(1\) is the exact Orlicz threshold.
		
		\item In one dimension, rotations convert the operator into an angular
		multiplier problem.  The endpoint question then becomes a question about the
		exact variation of the finite-part angular profile generated by \(\Omega\).
		
		\item Directional estimates in \cite{DQX0404} rely on a positive power gain
		when \(a>1/2\).  At \(a=1/2\) this gain disappears, suggesting a critical
		logarithmic replacement and a comparison with the Orlicz endpoint.
	\end{itemize}
	
	\subsection{Main results}
	
	We now state our results.  The first theorem is the structural input behind
	the rotational approach.  Parametrize \(\Sph\) by
	\(\theta(t)=(\cos t,\sin t)\), set \(q(t)=\Omega(\theta(t))\), and write
	\(q_e(t)=(q(t)+q(t+\pi))/2\), \(q_o(t)=(q(t)-q(t+\pi))/2\).  The finite-part
	formula for the rough singular integral identifies its multiplier with the
	angular profile
	\[
	m_\Omega(\xi,\eta)=\cM_\Omega(t)
	\quad\text{whenever}\quad
	\frac{(\xi,\eta)}{\abs{(\xi,\eta)}}=(\cos t,\sin t),
	\]
	where
	\[
	\cM_\Omega(t)
	:=
	c_{\rm rad}\int_{\T}q(\varphi)
	\left[
	-\log|\cos(\varphi-t)|
	-\frac{i\pi}{2}\sgn(\cos(\varphi-t))
	\right]\,d\sigma(\varphi),
	\]
	and \(c_{\rm rad}>0\) is fixed by the polar-coordinate normalization.  Thus the
	roughness of the kernel is transferred to the variation of a one-dimensional
	angular profile.  Once \(\cM_\Omega\in BV\), the rotational decomposition for
	angular multipliers applies to \(m_\Omega\).
	
	The following theorem gives the exact characterization.  Here \(H_\T\) denotes the
	periodic Hilbert transform.  We use the norm
	\[
	\norm F_{H^1(\T)}
	:=
	\norm F_{L^1(\T)}+\norm{H_\T F}_{L^1(\T)},
	\]
	where \(H_\T F\) is initially understood in the sense of distributions, and
	transport this norm to \(H^1(\Sph)\) through the above parametrization.  For an
	\(L^1\) profile, membership in \(BV(\Sph)\) means that its almost-everywhere
	defined equivalence class admits a \(BV\) representative.
	
	\begin{theorem}
		\label{thm:profile-characterization}
		Let \(\Omega\in L^1(\Sph)\) have mean zero. Then
		\[
		\cM_\Omega\in BV(\Sph)
		\quad\Longleftrightarrow\quad
		H_\T q_e\in L^1(\T).
		\]
		Equivalently,
		\[
		\cM_\Omega\in BV(\Sph)
		\quad\Longleftrightarrow\quad
		\Omega_e\in H^1(\Sph).
		\]
		More precisely,
		\[
		\Var_{\Sph}(\cM_\Omega)
		\simeq
		\norm{H_\T q_e}_{L^1(\T)}
		+
		\norm{q_o}_{L^1(\T)},
		\]
		with constants depending only on the normalizations.
	\end{theorem}
	
	This characterization is tailored to the method of rotations.  In the
	bilinear setting, Diestel, Grafakos, Honz\'ik, Si, and Terwilleger
	\cite{DGHST2011} combined the \(H^1\)-regularity of the even part with uniform
	bilinear Hilbert transform estimates in a restricted exponent region.
	Combining Theorem~\ref{thm:profile-characterization} with a \(BV\)/Stieltjes
	decomposition of angular multipliers and the uniform bilinear Hilbert transform
	theorem of Uraltsev and Warchalski \cite{UraltsevWarchalski2022} yields a
	criterion throughout the finite-exponent Banach range.
	
	\begin{theorem}
		\label{thm:rotational-main}
		Let \(\Omega\in L^1(\Sph)\) have mean zero, and write
		\(\Omega=\Omega_e+\Omega_o\) for its antipodal even and odd parts. If 
		\(\Omega_e\in H^1(\Sph)\), then, for every
		$
		1<p_1,p_2<\infty,\,
		1<p<\infty,\,
		\frac1p=\frac1{p_1}+\frac1{p_2},
		$
		the principal-value operator \(T_\Omega\) extends boundedly from
		\(L^{p_1}(\R)\times L^{p_2}(\R)\) to \(L^p(\R)\), and
		\[
		\norm{T_\Omega(f,g)}_{L^p}
		\le
		C_{p_1,p_2}
		\bigl(\norm{\Omega_e}_{H^1(\Sph)}
		+\norm{\Omega_o}_{L^1(\Sph)}\bigr)
		\norm f_{L^{p_1}}\norm g_{L^{p_2}}.
		\]
	\end{theorem}
	
	The \(L\log L\) result follows from Theorem \ref{thm:rotational-main} and the
	Riesz--Zygmund theorem for conjugate functions \cite{Zygmund1959}.  Recall that
	\(L\log L=L(\log L)^1\) for the Orlicz scale introduced above.
	
	\begin{theorem}\label{thm:LlogL-main}
		Let
		$
		1<p_1,p_2<\infty,\,
		1<p<\infty,\,
		\frac1p=\frac1{p_1}+\frac1{p_2}.
		$
		If  \(\Omega\in L\log L(\Sph)\) has mean zero, then
		\(T_\Omega\) extends to a bounded bilinear map
		$
		T_\Omega:L^{p_1}(\R)\times L^{p_2}(\R)\to L^p(\R),
		$
		and
		\[
		\norm{T_\Omega(f,g)}_{L^p}
		\le
		C_{p_1,p_2}\norm{\Omega}_{L\log L(\Sph)}
		\norm f_{L^{p_1}}\norm g_{L^{p_2}}.
		\]
	\end{theorem}
	
	\begin{theorem}\label{thm:Orlicz-sharp}
		Fix exponents as in Theorem \ref{thm:LlogL-main}, and let \(A_*(p_1,p_2)\)
		denote the infimum of all \(A\ge0\) such that
		\[
		\norm{T_\Omega}_{L^{p_1}\times L^{p_2}\to L^p}
		\le
		C_A\norm{\Omega}_{L(\log L)^A(\Sph)}
		\]
		for all mean-zero \(\Omega\in L(\log L)^A(\Sph)\). Then
		$
		A_*(p_1,p_2)=1.
		$
	\end{theorem}
	
	Theorem \ref{thm:Orlicz-sharp} shows that the \(L\log L\) assumption cannot be
	weakened inside the scale \(L(\log L)^A\).  The proof tests the logarithmic
	singularity of the angular profile against thin angular atoms and converts the
	resulting multiplier lower bound into an operator-norm lower bound.  In this
	sense the first part of the paper identifies the exact Orlicz threshold.
	
	We next turn to directional assumptions.  In the linear theory, Grafakos and
	Stefanov \cite{GrafakosStefanov1998} introduced a logarithmic condition which,
	in dimension \(n\), has the form
	\begin{equation}\label{eq:GS-general-intro}
		\sup_{\xi\in\mathbb S^{n-1}}
		\int_{\mathbb S^{n-1}}
		|\Omega(\theta)|
		\left(\log \frac1{|\theta\cdot\xi|}\right)^{1+\alpha}
		\,d\sigma(\theta)<\infty,
		\qquad \alpha>0.
	\end{equation}
	For bilinear rough singular integrals, boundedness under this type of
	Grafakos--Stefanov condition remains open.  In the one-dimensional case, the
	fractional directional assumption
	\begin{equation}\label{eq:Ka-intro}
		\sup_{\xi\in\Sph}
		\int_{\Sph}
		\frac{|\Omega(\theta)|}{|\theta\cdot\xi|^a}
		\,d\sigma(\theta)<\infty,
		\qquad \frac12<a<1.
	\end{equation}
	is a useful substitute.  We denote the supremum on the left-hand side of
	\eqref{eq:Ka-intro} by \(\norm{\Omega}_{\cK_a}\).
	The range \(a>1/2\) was treated in \cite{DQX0404}.  At \(a=1/2\), the dyadic
	estimates lose the power decay available above the critical index.  We
	therefore insert a logarithmic factor at the critical power.  The resulting
	class is stronger than the \(\Sph\)-version of \eqref{eq:GS-general-intro}, but
	it is the critical analogue of \eqref{eq:Ka-intro}.
	
	\begin{definition}\label{def:Klog}
		For \(\beta>0\), define
		\[
		\norm{\Omega}_{\cK_{1/2,\beta}}
		:=
		\sup_{\xi\in \Sph}
		\int_{\Sph}
		\frac{|\Omega(\theta)|}{|\theta\cdot\xi|^{1/2}}
		\left(\log\frac1{|\theta\cdot\xi|}\right)^\beta
		\,d\sigma(\theta).
		\]
		We write \(\Omega\in\cK_{1/2,\beta}(\Sph)\) when this quantity is
		finite.
	\end{definition}
	
	\begin{theorem}\label{thm:Klog-main}
		Let
		$
		1<p_1,p_2<\infty,\,
		1<p<\infty,\,
		\frac1p=\frac1{p_1}+\frac1{p_2}.
		$
		Assume that \(\Omega\) has mean zero and belongs to
		\(\cK_{1/2,\beta}(\Sph)\). If
		$
		\beta>\frac32\max\{p_1,p_1',p_2,p_2'\}-1,
		$
		then \(T_\Omega\) extends boundedly from
		\(L^{p_1}(\R)\times L^{p_2}(\R)\) to \(L^p(\R)\), with
		\[
		\norm{T_\Omega(f,g)}_{L^p}
		\le
		C_{p_1,p_2,\beta}
		\norm{\Omega}_{\cK_{1/2,\beta}}
		\norm f_{L^{p_1}}\norm g_{L^{p_2}}.
		\]
	\end{theorem}
	
	At the formal symmetric boundary \(p_1=p_2=2\), one has \(p=1\), so this
	point lies outside the range \(p>1\) of Theorem \ref{thm:Klog-main}.
	Nevertheless, both the displayed threshold and the high-frequency
	\(L^2\times L^2\to L^1\) summability requirement below lead, in the limit, to
	\(\beta>2\) at this boundary.  We record this only as a limiting benchmark:
	neither the boundary case nor the sharpness of the displayed threshold is
	claimed here.  The logarithmic factor compensates for the loss of the power
	gain available when \(a>1/2\).  If
	\(\cK_a(\Sph)\) denotes the class determined by \eqref{eq:Ka-intro}, then
	\(\cK_a(\Sph)\subset\cK_{1/2,\beta}(\Sph)\) for every \(a>1/2\) and
	\(\beta>0\).  Thus Theorem \ref{thm:Klog-main} extends the boundedness mechanism
	of \cite{DQX0404} to a logarithmically strengthened condition at criticality.
	
	\begin{theorem}\label{thm:incomparability}
		Let \(\alpha>0\) and \(\beta>0\). Then the spaces
		$L(\log L)^\alpha(\Sph)$ and $\cK_{1/2,\beta}(\Sph)$
		are incomparable; namely
		\[
		L(\log L)^\alpha(\Sph)\not\subset\cK_{1/2,\beta}(\Sph),
		\qquad
		\cK_{1/2,\beta}(\Sph)\not\subset L(\log L)^\alpha(\Sph).
		\]
	
	\end{theorem}
	
		In particular, \(L\log L(\Sph)\) and
	\(\cK_{1/2,\beta}(\Sph)\) are  incomparable. Theorem \ref{thm:incomparability} shows that the two endpoint criteria measure
	different singular behavior.  The \(L\log L\) condition controls the global
	size of the angular kernel, while \(\cK_{1/2,\beta}\) imposes weighted
	integrability near every orthogonal direction.  Neither condition dominates the
	other.
	
	The main contributions of this paper are as follows:
	\begin{itemize}
		\item We characterize exactly when the finite-part angular profile has
		bounded variation.  This gives a rotational criterion throughout the
		finite-exponent Banach range
		(Theorems~\ref{thm:profile-characterization}
		and~\ref{thm:rotational-main}).
		
		\item We prove boundedness under the endpoint hypothesis
		\(\Omega\in L\log L(\Sph)\) and show that the logarithmic exponent \(1\)
		cannot be lowered within the scale \(L(\log L)^A\)
		(Theorems~\ref{thm:LlogL-main} and~\ref{thm:Orlicz-sharp}).
		
		\item We introduce a logarithmically corrected directional condition at the
		critical index \(a=1/2\), prove the corresponding operator bound, and show
		that this directional class is incomparable with every positive logarithmic
		Orlicz class
		(Theorems~\ref{thm:Klog-main} and~\ref{thm:incomparability}).
	\end{itemize}
	
	To demonstrate our novelty, we now explain  the main ideas in this paper. The two endpoint criteria require different arguments.  For the \(L\log L\)
	result, the starting point is the finite-part multiplier formula
	\cite[Proposition~4.2.3]{GrafakosClassical}.  After splitting \(\Omega\) into
	its antipodal even and odd parts, the formula separates into two pieces:
	the derivative of the logarithmic part is determined by the periodic Hilbert
	transform of the even part, while the derivative of the sign part is determined
	directly by the odd part.  The \(BV\) regularity of \(\cM_\Omega\) is therefore
	reduced to a one-dimensional conjugate-function problem.  The F.~and M.~Riesz
	theorem gives the converse implication in the
	\(BV\)--\(H^1\) characterization, while the Riesz--Zygmund estimate provides
	the endpoint \(L\log L\to L^1\) control; see
	\cite[Chapter~VII]{Zygmund1959}.  The corresponding quantitative profile
	estimate is recorded in Proposition~\ref{prop:profile-bv}.
	
	We then pass from angular regularity to operator bounds.  A \(BV\) profile is a
	Stieltjes superposition of step functions.  The corresponding bilinear
	multipliers are half-plane multipliers, controlled uniformly by bilinear
	Hilbert transform estimates.  This lets us combine the rotational argument of
	\cite{DGHST2011} with the uniform theorem of
	\cite{UraltsevWarchalski2022} without imposing absolute continuity on the
	angular profile.
	
	The directional endpoint is proved by a separate mechanism.  The power gain
	used for \(a>1/2\) in \cite{DQX0404} disappears at \(a=1/2\).  The logarithmic
	factor in \(\cK_{1/2,\beta}\) replaces it by the annular decay
	\(2^{-j/2}j^{-\beta}\).  Using the product-wavelet framework of
	\cite{GrafakosHeHonzikPark2023}, we convert this decay into the estimate
	\[
	\norm{T_j}_{L^2\times L^2\to L^1}\lesssim j^{1-\beta}.
	\]
	We then interpolate this estimate with the polynomial bound in
	Lemma~\ref{lem:poly-growth}, whose proof uses the shifted-operator argument
	from \cite[the proof of Proposition~4]{dosidis_multilinear_2024}, and
	apply Sagher's theorem \cite{Sagher1969} to obtain the required summability and
	the stated condition on \(\beta\).
	
	The organization of this paper is as follows.  In Section 2, we prove the
	finite-part angular-profile characterization and the corresponding rotational
	criterion, namely Theorems~\ref{thm:profile-characterization}
	and~\ref{thm:rotational-main}.  Section 2  contains the proof of Theorems~\ref{thm:LlogL-main} and~\ref{thm:Orlicz-sharp}, the
	\(L\log L\) endpoint theorem and its sharpness in the Orlicz scale.
	Section 3 is devoted
	to the proof of critical directional endpoint, Theorem~\ref{thm:Klog-main}.  The proof
	of the incomparability theorem, Theorem~\ref{thm:incomparability}, is given in
	Section 4.  Throughout this paper, the letter \(C\) denotes a positive constant,
	not necessarily the same at each occurrence, but always independent of the
	functions under consideration and of the relevant summation and truncation
	parameters.
	\vspace{0.3cm}
	
	\section{Boundedness under the sharp \(L\log L\) kernel condition}
	
	\subsection{Finite-part profiles and the rotational criterion}
	
We begin with the finite-part formula that identifies the multiplier of the
rough kernel; see \cite[Proposition~4.2.3]{GrafakosClassical}.  We then prove
the \(BV\) multiplier estimate needed for the rotational theorem.
	
	\begin{lemma}
		\label{lem:finite-part}
		Let \(\Omega\in L^\infty(\Sph)\) have mean zero.  Then
		\[
		m_\Omega(\xi,\eta)=\cM_\Omega
		\left(\frac{(\xi,\eta)}{(\xi^2+\eta^2)^{1/2}}\right),
		\qquad (\xi,\eta)\ne(0,0),
		\]
		where
		\[
		\cM_\Omega(u)
		=
		c\int_{\Sph}\Omega(\theta)
		\left[-\log|\theta\cdot u|-\frac{i\pi}{2}\sgn(\theta\cdot u)\right]d\sigma(\theta),
		\qquad u\in \Sph,
		\]
		and \(c>0\) depends only on the normalizations.
	\end{lemma}
	
	\begin{proof}
		We use the Fourier transform convention
		\[
		\widehat F(\zeta)
		=
		\int_{\R^2}F(y)e^{-2\pi i y\cdot\zeta}\,dy,
		\qquad \zeta\in\R^2.
		\]
		Since \(d\sigma=d\varphi/(2\pi)\), polar coordinates give
		\(dy=2\pi r\,dr\,d\sigma(\theta)\). For \(\delta>0\), consider the
		Abel-regularized kernel
		\[
		K_{\Omega,\delta}(y)
		=
		e^{-\delta|y|}\frac{\Omega(y/|y|)}{|y|^2},
		\]
		interpreted at the origin by radial principal value. More precisely, for
		\(\Psi\in\mathcal S(\R^2)\), set
		\[
		\begin{aligned}
			\langle K_{\Omega,\delta},\Psi\rangle
			=2\pi\int_{\Sph}\Omega(\theta)\bigg[
			&\int_0^1 e^{-\delta r}
			\frac{\Psi(r\theta)-\Psi(0)}r\,dr\\
			&+\int_1^\infty e^{-\delta r}
			\frac{\Psi(r\theta)}r\,dr
			\bigg]d\sigma(\theta).
		\end{aligned}
		\]
		The subtraction in the first integral is permitted by
		\(\int_{\Sph}\Omega\,d\sigma=0\). Taylor's theorem at the origin and
		the rapid decay of \(\Psi\) show that the expression is well defined.
		They also give, by dominated convergence,
		\(K_{\Omega,\delta}\to K_\Omega\) in \(\mathcal S'(\R^2)\) as
		\(\delta\downarrow0\), where \(K_\Omega\) is the finite-part distribution
		determined by the original homogeneous kernel.
		
		Let \(\zeta=(\xi,\eta)\ne0\). After using the angular cancellation, the
		multiplier of the regularized kernel can be written as
		\[
		\begin{aligned}
			m_{\Omega,\delta}(\zeta)
			=2\pi\int_{\Sph}\Omega(\theta)
			\int_0^\infty
			\left(
			e^{-(\delta+2\pi i\theta\cdot\zeta)r}-e^{-r}
			\right)\frac{dr}{r}\,d\sigma(\theta).
		\end{aligned}
		\]
		The term \(e^{-r}\) is independent of \(\theta\), so its insertion does
		not change the multiplier. The inner integral is now absolutely convergent
		at both endpoints. For complex numbers \(a,b\) with positive real parts,
		Frullani's identity gives
		\[
		\int_0^\infty(e^{-ar}-e^{-br})\,\frac{dr}{r}
		=
		\log\frac ba.
		\]
		Taking \(a=\delta+2\pi i\theta\cdot\zeta\), \(b=1\), and the principal
		branch of the logarithm, we obtain
		\[
		m_{\Omega,\delta}(\zeta)
		=
		-2\pi\int_{\Sph}\Omega(\theta)
		\log\bigl(\delta+2\pi i\theta\cdot\zeta\bigr)\,d\sigma(\theta).
		\]
		
		For \(s\ne0\),
		\[
		-\log(\delta+2\pi i s)
		\longrightarrow
		-\log(2\pi|s|)-\frac{i\pi}{2}\sgn(s)
		\qquad(\delta\downarrow0).
		\]
		This limit may be passed through the angular integral. Indeed, if
		\(u=\zeta/|\zeta|\), then
		\(\theta\cdot\zeta=|\zeta|\theta\cdot u\), and the only angular
		singularities are translates of \(\log|\cos\varphi|\), which belongs to
		\(L^1(\T)\). The convergence is in fact local on the whole frequency
		plane, including the origin. To see this, put
		\(t=\delta/(2\pi|\zeta|)\) for \(\zeta\ne0\). The mean-zero condition gives
		\[
		m_{\Omega,\delta}(\zeta)
		=
		-2\pi\int_{\Sph}\Omega(\theta)
		\bigl[\log(t+i\theta\cdot u)-\log(t+1)\bigr]\,d\sigma(\theta).
		\]
		Moreover,
		\[
		\sup_{t>0}\sup_{u\in\Sph}
		\int_{\Sph}
		\bigl|\log(t+i\theta\cdot u)-\log(t+1)\bigr|\,d\sigma(\theta)<\infty.
		\]
		By rotation invariance, it is enough to take \(\theta\cdot u=\cos\varphi\).
		For \(0<t\le1\), the estimate follows from
		\(\log|\cos\varphi|\in L^1(\T)\) and the boundedness of the argument; for
		\(t\ge1\), the logarithm of \((t+i\cos\varphi)/(t+1)\) is uniformly
		bounded. Hence \(m_{\Omega,\delta}\) is uniformly bounded on \(\R^2\)
		(after setting \(m_{\Omega,\delta}(0)=0\)). Dominated convergence now gives
		convergence in \(L^1_{\mathrm{loc}}(\R^2)\). Since
		\(K_{\Omega,\delta}\to K_\Omega\) in \(\mathcal S'\), the locally
		integrable limit represents \(\widehat K_\Omega\). Thus
		\[
		m_\Omega(\zeta)
		=
		2\pi\int_{\Sph}\Omega(\theta)
		\left[
		-\log\bigl(2\pi|\theta\cdot\zeta|\bigr)
		-\frac{i\pi}{2}\sgn(\theta\cdot\zeta)
		\right]d\sigma(\theta).
		\]
		Finally,
		\[
		-\log\bigl(2\pi|\theta\cdot\zeta|\bigr)
		=
		-\log(2\pi|\zeta|)-\log|\theta\cdot u|.
		\]
		The first term is independent of \(\theta\) and disappears by the
		mean-zero condition; also
		\(\sgn(\theta\cdot\zeta)=\sgn(\theta\cdot u)\). Consequently,
		\[
		m_\Omega(\xi,\eta)
		=
		2\pi\int_{\Sph}\Omega(\theta)
		\left[
		-\log|\theta\cdot u|
		-\frac{i\pi}{2}\sgn(\theta\cdot u)
		\right]d\sigma(\theta),
		\qquad
		u=\frac{(\xi,\eta)}{(\xi^2+\eta^2)^{1/2}}.
		\]
		Thus with the present Fourier and surface-measure normalizations, the
		constant in the statement is \(c=2\pi\).
	\end{proof}
	
	We next formulate the cone-decomposition argument of \cite{DGHST2011} for
	\(BV\) angular profiles.  In place of an absolutely continuous representative,
	we use its Stieltjes derivative together with uniform bounds for line-sign
	multipliers.
	
	We begin by fixing the normalization of the bilinear Hilbert transform used
	below.  For
	\(\beta\in\R\setminus\{0,1\}\), define initially on Schwartz functions
	\[
	\operatorname{BHT}_{\beta}(f,g)(x)
	:=
	\operatorname{p.v.}\int_\R
	f(x-t)g(x-\beta t)\,\frac{dt}{t}.
	\]
	Up to normalization, \(\operatorname{BHT}_\beta\) is the bilinear multiplier
	with symbol \(\sgn(\xi+\beta\eta)\).  The uniform theorem of Uraltsev and
	Warchalski, together with the trilinear symmetries used below, gives bounds
	that are uniform in the direction throughout the finite-exponent Banach range
	considered here; see \cite{UraltsevWarchalski2022}, building on
	\cite{LaceyThiele1997,LaceyThiele1999,GrafakosLi2004,Li2006,Thiele2002}.
	
	\begin{lemma}
		\label{lem:line-sign}
		Let \(1<p_1,p_2<\infty\), \(1<p<\infty\), and
		\(1/p=1/p_1+1/p_2\).  For every nonzero \(v=(v_1,v_2)\), let \(T_v\) be the
		bilinear multiplier with symbol \(\sgn(v_1\xi+v_2\eta)\).  Then
		\[
		\norm{T_v}_{L^{p_1}(\R)\times L^{p_2}(\R)\to L^p(\R)}
		\le C_{p_1,p_2},
		\]
		uniformly in \(v\).
	\end{lemma}
	
	\begin{proof}
		Let \(p_3=p'\) and consider the associated trilinear form on the hyperplane
		\(\xi_1+\xi_2+\xi_3=0\).  The line-sign symbol has coefficient triple
		\((v_1,v_2,0)\).  Adding a common scalar to the three coefficients leaves the
		symbol unchanged on this hyperplane, whereas permuting the variables merely
		permutes the exponents \(p_1,p_2,p_3\).
		
		Suppose first that \(v_1,v_2,0\) are pairwise distinct.  Subtracting the middle
		coefficient, permuting the variables, and rescaling all coefficients by a
		nonzero scalar reduces the triple, up to an overall harmless sign, to
		\[
		(1,-\beta,0),
		\qquad 0<\beta\le1.
		\]
		Thus the permuted trilinear form is that of the bilinear Hilbert transform
		with symbol \(\sgn(\xi-\beta\eta)\), equivalently
		\(\operatorname{BHT}_{-\beta}\) in the displayed convention.  The uniform
		theorem of Uraltsev and Warchalski \cite{UraltsevWarchalski2022} applies.  Every
		permuted input exponent belongs to the triple \(p_1,p_2,p'\), hence is greater
		than one; the corresponding output exponent is the conjugate of the remaining
		member of this triple and is also greater than one.  Taking the maximum over
		the finitely many permutations yields a constant depending only on
		\(p_1,p_2\).
		
		If \(v_1=0\), \(v_2=0\), or \(v_1=v_2\), the operator is, respectively,
		\(fHg\), \((Hf)g\), or \(H(fg)\), up to harmless signs.  These cases follow
		from H\"older's inequality and the boundedness of the linear Hilbert transform;
		the last case uses \(p>1\).
	\end{proof}
	
	It follows that the half-plane multipliers
	\[
	\one_{\{v_1\xi+v_2\eta>0\}}
	=
	\frac12\bigl(1+\sgn(v_1\xi+v_2\eta)\bigr)
	\]
	are also bounded from \(L^{p_1}\times L^{p_2}\) to \(L^p\), uniformly in
	\(v\ne0\).
	
	\begin{lemma}\label{lem:one-chart-bv}
		Let \(b\in BV(\R)\cap L^\infty(\R)\).  Each of the four model multipliers
		\[
		\one_{\{\xi>0\}}b(\eta/\xi),\quad
		\one_{\{\xi<0\}}b(\eta/(-\xi)),\quad
		\one_{\{\eta>0\}}b(\xi/\eta),\quad
		\one_{\{\eta<0\}}b(\xi/(-\eta))
		\]
		defines an operator whose norm is at most
		\[
		C_{p_1,p_2}\bigl(\norm b_\infty+\Var_\R(b)\bigr).
		\]
	\end{lemma}
	
	\begin{proof}
		We consider \(m(\xi,\eta)=\one_{\{\xi>0\}}b(\eta/\xi)\).  A \(BV\) function on
		\(\R\) has finite limits at \(\pm\infty\)
		\cite[Theorem~3.27(c)]{Folland1999}.  Choose the right-continuous
		representative of \(b\), and let \(\nu=db\) be the associated
		Lebesgue--Stieltjes measure.  By the correspondence between
		one-dimensional \(BV\) functions and finite Borel measures, and by the
		identification of the total variation measure
		\cite[Theorem~3.29]{Folland1999},
		\[
		|\nu|(\R)=\Var_\R(b),
		\]
		and
		\[
		b(s)=b(-\infty)+\nu((-\infty,s]).
		\]
		Up to changes on multiplier-null boundary sets, this gives
		\[
		\one_{\{\xi>0\}}b(\eta/\xi)
		=
		b(-\infty)\one_{\{\xi>0\}}
		+
		\int_\R
		\one_{\{\xi>0\}}\one_{\{\eta-t\xi>0\}}\,d\nu(t).
		\]
		Let \(P_+\) denote the Fourier projection with symbol
		\(\one_{\{\xi>0\}}\), and let \(\Pi_t\) denote the half-plane bilinear
		multiplier with symbol \(\one_{\{\eta-t\xi>0\}}\).  The multiplier
		\[
		\one_{\{\xi>0\}}\one_{\{\eta-t\xi>0\}}
		\]
		acts as
		\[
		(f,g)\mapsto \Pi_t(P_+f,g).
		\]
		The projection \(P_+\) is bounded on \(L^{p_1}\), and the half-plane
		consequence of Lemma \ref{lem:line-sign} bounds \(\Pi_t\) uniformly in \(t\).
		Thus these terms have operator norm \(\le C_{p_1,p_2}\), uniformly in \(t\).
		The constant term
		\(b(-\infty)P_+f\cdot g\) is controlled by H\"older's inequality and the
		\(L^{p_1}\)-boundedness of \(P_+\).  Estimating the operator-valued Stieltjes
		integral in total variation yields
		\[
		\norm{T_m}
		\le
		C_{p_1,p_2}\bigl(|b(-\infty)|+|\nu|(\R)\bigr)
		\le
		C_{p_1,p_2}\bigl(\norm b_\infty+\Var_\R(b)\bigr).
		\]
		
		For \(m(\xi,\eta)=\one_{\{\xi<0\}}b(\eta/(-\xi))\), we use the projection
		\(P_-\) and the half-plane symbols \(\one_{\{\eta+t\xi>0\}}\).  For the two
		\(\eta\)-charts, use the corresponding projections in the second input and
		the half-plane symbols \(\one_{\{\xi-t\eta>0\}}\) or
		\(\one_{\{\xi+t\eta>0\}}\).  The same argument yields the stated bound in
		each case.
	\end{proof}
	
	\begin{proposition}\label{prop:bv-angular}
		Let \(a\in BV(\Sph)\cap L^\infty(\Sph)\), and set
		\[
		m_a(\xi,\eta)=a\left(\frac{(\xi,\eta)}{(\xi^2+\eta^2)^{1/2}}\right).
		\]
		Then
		\[
		\norm{T_{m_a}}_{L^{p_1}\times L^{p_2}\to L^p}
		\le
		C_{p_1,p_2}\bigl(\norm a_\infty+\Var_{\Sph}(a)\bigr).
		\]
	\end{proposition}
	
	\begin{proof}
		Fix \(0<c<2^{-1/2}\), and set
		\[
		U_1=\{\theta\in \Sph:\theta_1>c\},\quad
		U_2=\{\theta\in \Sph:\theta_1<-c\},
		\]
		\[
		U_3=\{\theta\in \Sph:\theta_2>c\},\quad
		U_4=\{\theta\in \Sph:\theta_2<-c\}.
		\]
		Then \(\Sph=\bigcup_{j=1}^4U_j\).  Choose
		\(\chi_j\in C^\infty(\Sph)\) such that
		\[
		0\le\chi_j\le1,\qquad
		\operatorname{supp}\chi_j\subset U_j,\qquad
		\sum_{j=1}^4\chi_j=1,
		\]
		and put \(a_j=\chi_j a\).  Since the cutoffs are fixed,
		\[
		\norm{a_j}_\infty+\Var_{\Sph}(a_j)
		\le
		C\bigl(\norm a_\infty+\Var_{\Sph}(a)\bigr).
		\]
		By linearity, it suffices to prove the required estimate for each \(a_j\).
		
		For \(a_1\), define
		\[
		b_1(s)=a_1\left(\frac{(1,s)}{(1+s^2)^{1/2}}\right),
		\qquad s\in\R.
		\]
		Away from the origin in frequency space,
		\[
		a_1\left(\frac{(\xi,\eta)}{(\xi^2+\eta^2)^{1/2}}\right)
		=
		\one_{\{\xi>0\}}b_1(\eta/\xi).
		\]
		Indeed, the right-hand side is the slope parametrization of the right
		half-circle and vanishes outside this chart because
		\(\operatorname{supp}a_1\subset U_1\).  Since
		\(s\mapsto(1,s)/(1+s^2)^{1/2}\) is a monotone \(C^\infty\)
		parametrization of the right half-circle,
		\[
		\norm{b_1}_\infty+\Var_\R(b_1)
		\le
		\norm{a_1}_\infty+\Var_{\Sph}(a_1)
		\le
		C\bigl(\norm a_\infty+\Var_{\Sph}(a)\bigr).
		\]
		Lemma \ref{lem:one-chart-bv} applies to this model multiplier.
		
		The remaining charts are treated identically after a change of coordinates:
		\[
		b_2(s)=a_2\left(\frac{(-1,s)}{(1+s^2)^{1/2}}\right),\quad
		b_3(s)=a_3\left(\frac{(s,1)}{(1+s^2)^{1/2}}\right),\quad
		b_4(s)=a_4\left(\frac{(s,-1)}{(1+s^2)^{1/2}}\right).
		\]
		They give, respectively, the model symbols
		\[
		\one_{\{\xi<0\}}b_2(\eta/(-\xi)),\qquad
		\one_{\{\eta>0\}}b_3(\xi/\eta),\qquad
		\one_{\{\eta<0\}}b_4(\xi/(-\eta)).
		\]
		Each \(b_j\) satisfies the same \(BV\) bound as \(b_1\).  Summing the four
		operator bounds completes the proof.
	\end{proof}
	
	\subsection{The \texorpdfstring{\(BV\)}{BV} characterization}
	
	Parametrize \(\Sph\) by \(\theta(\varphi)=(\cos\varphi,\sin\varphi)\), and
	write \(d\sigma=d\varphi/(2\pi)\).
	Set
	\[
	K_0(t)=-\log|\cos t|,
	\qquad
	S_0(t)=\sgn(\cos t).
	\]
	
	We use the following normalization of the periodic Hilbert transform.  For
	smooth \(2\pi\)-periodic \(F\),
	\[
	H_\T F(t)
	:=
	\pv\int_\T F(\varphi)\cot\frac{t-\varphi}{2}\,d\sigma(\varphi)
	=
	\lim_{\varepsilon\downarrow0}
	\int_{\{|t-\varphi|>\varepsilon\}}
	F(\varphi)\cot\frac{t-\varphi}{2}\,d\sigma(\varphi).
	\]
	Equivalently,
	\[
	\widehat{H_\T F}(n)=-i\,\sgn(n)\widehat F(n),
	\qquad n\in\mathbb Z.
	\]
	
	\begin{lemma}\label{lem:profile-derivatives}
		Let \(q_e,q_o\in L^1(\T)\) satisfy
		\[
		q_e(t+\pi)=q_e(t),
		\qquad
		q_o(t+\pi)=-q_o(t).
		\]
		Define
		\[
		L_e(t)=\int_\T q_e(\varphi)K_0(\varphi-t)\,d\sigma(\varphi),
		\qquad
		S_o(t)=\int_\T q_o(\varphi)S_0(\varphi-t)\,d\sigma(\varphi).
		\]
		Then, in \(\mathcal D'(\T)\),
		\[
		DL_e=c_H(H_\T q_e)(\,\cdot+\pi/2\,),
		\qquad
		DS_o=c_S q_o(\,\cdot+\pi/2\,),
		\]
		where \(c_H,c_S\ne0\) are absolute constants determined by the normalization
		of \(d\sigma\).  Consequently, if \(H_\T q_e\in L^1(\T)\), then
		\(L_e,S_o\in W^{1,1}(\T)\), and
		\[
		\Var_\T(L_e)+\Var_\T(S_o)
		\simeq
		\norm{H_\T q_e}_{L^1(\T)}
		+
		\norm{q_o}_{L^1(\T)}.
		\]
	\end{lemma}
	
	\begin{proof}
		The distributional identity \(DK_0=\pv\tan t\), together with
		\[
		\tan t
		=
		-\frac12
		\left(
		\cot\frac{t-\pi/2}{2}
		+
		\cot\frac{t+\pi/2}{2}
		\right)
		\]
		shows that \(DL_e\) is a nonzero constant multiple of a translate of
		\(H_\T q_e\).  The two shifted cotangent terms coincide by the
		\(\pi\)-periodicity of \(q_e\).
		
		For the sign kernel,
		\[
		S_0(t)=2\one_{(-\pi/2,\pi/2)}(t)-1
		\]
		on \(\T\).  Hence
		\[
		S_o(t)=
		2\int_{t-\pi/2}^{t+\pi/2}q_o(\varphi)\,d\sigma(\varphi)
		-\int_\T q_o\,d\sigma,
		\]
		and differentiating this absolutely continuous function gives
		\[
		DS_o(t)
		=
		\frac1\pi
		\bigl(q_o(t+\pi/2)-q_o(t-\pi/2)\bigr)
		=
		\frac2\pi q_o(t+\pi/2).
		\]
		This proves both identities and the variation estimate.
	\end{proof}
	
	\begin{proof}[Proof of Theorem \ref{thm:profile-characterization}]
		Put \(q(t)=\Omega(\theta(t))\) and decompose
		\(q=q_e+q_o\) with respect to the antipodal map \(t\mapsto t+\pi\).  Since
		\(K_0\) is
		antipodally even and \(S_0\) is antipodally odd,
		\[
		\cM_\Omega
		=
		c_{\rm rad}
		\left(
		L_e-\frac{i\pi}{2}S_o
		\right),
		\]
		with \(L_e,S_o\) as in Lemma \ref{lem:profile-derivatives}.
		
		If \(H_\T q_e\in L^1(\T)\), Lemma \ref{lem:profile-derivatives} gives
		\(L_e,S_o\in W^{1,1}(\T)\), and hence \(\cM_\Omega\in BV(\T)\), with
		\[
		\Var_\T(\cM_\Omega)
		\le
		C
		\bigl(
		\norm{H_\T q_e}_{L^1(\T)}
		+
		\norm{q_o}_{L^1(\T)}
		\bigr).
		\]
		
		Conversely, assume that \(\cM_\Omega\in BV(\T)\).  The antipodal
		projections
		\[
		P_eF(t)=\frac{F(t)+F(t+\pi)}2,
		\qquad
		P_oF(t)=\frac{F(t)-F(t+\pi)}2
		\]
		do not increase variation.  It follows that \(L_e\in BV(\T)\) and
		\(S_o\in BV(\T)\).  The identity for \(DS_o\) in Lemma
		\ref{lem:profile-derivatives} gives
		\[
		\norm{q_o}_{L^1(\T)}\le C\Var_\T(\cM_\Omega).
		\]
		It remains to prove that \(H_\T q_e\in L^1\).  Since \(L_e\in BV\), the
		distribution \(DL_e\) is a finite complex measure.  By Lemma
		\ref{lem:profile-derivatives}, after undoing the fixed translation and the
		nonzero normalization constant, \(H_\T q_e\) is represented by a finite
		measure \(\mu\).  Moreover,
		\[
		\norm{\mu}_{\mathcal M(\T)}
		\le C\norm{DL_e}_{\mathcal M(\T)}
		= C\Var_\T(L_e)
		\le C\Var_\T(\cM_\Omega).
		\]
		
		With the Fourier normalization above,
		\[
		\widehat\mu(n)=-i\,\sgn(n)\widehat{q_e}(n),
		\qquad n\ne0.
		\]
		Hence the finite measure \(\nu=\mu+i q_e\,d\sigma\) has
		\(\widehat\nu(n)=0\) for every \(n>0\).  The F. and M. Riesz theorem,
		applied to complex measures with one-sided vanishing Fourier spectrum
		\cite[Chapter VII]{Zygmund1959}, implies that \(\nu\) is absolutely
		continuous with respect to \(d\sigma\).  Since \(q_e\,d\sigma\) is
		absolutely continuous, \(\mu\) is absolutely continuous as well.  Its total
		variation is therefore the \(L^1\)-norm of its density.  Thus
		\(H_\T q_e\in L^1(\T)\), and the preceding measure estimate gives
		\[
		\norm{H_\T q_e}_{L^1(\T)}
		\le
		C\Var_\T(\cM_\Omega).
		\]
		
		Combining this estimate with the bound for \(q_o\) and the preceding upper
		estimate gives
		\[
		\Var_\T(\cM_\Omega)
		\simeq
		\norm{H_\T q_e}_{L^1(\T)}
		+
		\norm{q_o}_{L^1(\T)}.
		\]
		Under the parametrization \(\theta(t)=(\cos t,\sin t)\), this condition is
		precisely \(\Omega_e\in H^1(\Sph)\).
	\end{proof}
	
	\begin{proof}[Proof of Theorem \ref{thm:rotational-main}]
		Theorem \ref{thm:profile-characterization} gives
		\(\cM_\Omega\in BV(\Sph)\) and
		\[
		\Var(\cM_\Omega)
		\le
		C
		\bigl(
		\norm{\Omega_e}_{H^1(\Sph)}
		+
		\norm{\Omega_o}_{L^1(\Sph)}
		\bigr).
		\]
		The same decomposition also gives
		\[
		\norm{\cM_\Omega}_{L^1}
		\le
		C\bigl(
		\norm{\Omega_e}_{L^1(\Sph)}
		+
		\norm{\Omega_o}_{L^1(\Sph)}
		\bigr),
		\]
		because \(K_0\in L^1(\T)\) and \(S_0\in L^\infty(\T)\subset L^1(\T)\).
		Since
		\(\norm{F}_{L^\infty(\T)}\le C(\norm{F}_{L^1(\T)}+\Var_\T(F))\) for
		\(F\in BV(\T)\), we also have
		\[
		\norm{\cM_\Omega}_{L^\infty}
		+
		\Var(\cM_\Omega)
		\le
		C
		\bigl(
		\norm{\Omega_e}_{H^1(\Sph)}
		+
		\norm{\Omega_o}_{L^1(\Sph)}
		\bigr).
		\]
		When \(\Omega\) is bounded, Lemma \ref{lem:finite-part} identifies
		\(T_\Omega\) with the angular multiplier associated with
		\(\cM_\Omega\), and Proposition \ref{prop:bv-angular} gives the
		desired estimate.
		
		For general \(\Omega\), let \(\rho_k\) be a smooth periodic approximate
		identity and set \(q_{e,k}=\rho_k*q_e\).  Since \(q_e\) is \(\pi\)-periodic
		and has mean zero, so is \(q_{e,k}\).  Moreover, \(q_{e,k}\) is bounded and
		converges to \(q_e\) in \(H^1(\T)\), because \(H_\T\) commutes with
		convolution.  For \(z\in\C\), put
		\[
		\tau_k(z)=
		\begin{cases}
			z\min\{1,k/|z|\},&z\ne0,\\
			0,&z=0,
		\end{cases}
		\qquad q_{o,k}(t)=\tau_k(q_o(t)).
		\]
		The map \(\tau_k\) is odd, so \(q_{o,k}(t+\pi)=-q_{o,k}(t)\); in particular,
		it has mean zero.  Moreover, \(q_{o,k}\) is bounded and converges to \(q_o\)
		in \(L^1(\T)\).  Define
		\(\Omega_k(\theta(t))=q_{e,k}(t)+q_{o,k}(t)\).  Applying the bounded case to
		\(\Omega_k-\Omega_\ell\) shows that \(T_{\Omega_k}\) is Cauchy in operator
		norm; denote its limit by \(T\).
		
		It remains to identify \(T\) with the principal-value operator associated with
		\(\Omega\).  For \(f,g,h\in\mathcal S(\R)\), set
		\[
		\Phi_{f,g,h}(y_1,y_2)=\int_\R f(x-y_1)g(x-y_2)\overline{h(x)}\,dx.
		\]
		For a mean-zero \(\Lambda\in L^1(\Sph)\), define its finite-part distribution by
		\begin{equation}\label{eq:finite-part-L1}
			\begin{aligned}
				\langle \mathcal K_\Lambda,\Psi\rangle
				=
				c_{\rm rad}\int_{\Sph}\Lambda(\theta)\bigg[
				&\int_0^1\frac{\Psi(r\theta)-\Psi(0)}{r}\,dr\\
				&+\int_1^\infty\frac{\Psi(r\theta)}{r}\,dr
				\bigg]d\sigma(\theta),
			\end{aligned}
			\qquad \Psi\in\mathcal S(\R^2).
		\end{equation}
		Here \(c_{\rm rad}=2\pi\) under the normalization \(\sigma(\Sph)=1\), as in
		Lemma~\ref{lem:finite-part}.
		Taylor expansion at the origin and rapid decay at infinity show that the
		bracketed function is bounded on \(\Sph\).  Hence \(\mathcal K_{\Lambda_j}\to
		\mathcal K_\Lambda\) in \(\mathcal S'(\R^2)\) whenever \(\Lambda_j\to
		\Lambda\) in \(L^1(\Sph)\).
		
		We next verify that this distribution agrees with the radial principal value
		for every mean-zero \(\Lambda\in L^1(\Sph)\), not only for bounded
		\(\Lambda\).  For \(f,g\in\mathcal S(\R)\), write
		\(F_x(y)=f(x-y_1)g(x-y_2)\).  Polar coordinates and cancellation give, for
		\(0<\varepsilon<1\),
		\[
		\begin{aligned}
			T_{\Lambda,\varepsilon}(f,g)(x)
			=c_{\rm rad}\int_{\Sph}\Lambda(\theta)
			\bigg[
			&\int_\varepsilon^1
			\frac{F_x(r\theta)-F_x(0)}r\,dr\\
			&+\int_1^\infty\frac{F_x(r\theta)}r\,dr
			\bigg]d\sigma(\theta).
		\end{aligned}
		\]
		The mean-value theorem bounds the first integrand uniformly near \(r=0\).  In
		the second integral, at least one component of \(\theta\) has absolute value
		at least \(2^{-1/2}\), so the Schwartz decay is uniform in \(\theta\).  The
		bracket therefore converges uniformly in \(\theta\) as
		\(\varepsilon\downarrow0\), and the radial principal value exists for every
		\(x\).  The corresponding Schwartz-seminorm estimates justify pairing with
		\(h\) and applying Fubini to the subtracted formula.  Thus, for every
		mean-zero \(\Lambda\in L^1(\Sph)\),
		\begin{equation}\label{eq:finite-part-pv-pairing}
			\langle T_{\Lambda}(f,g),h\rangle
			=
			\langle\mathcal K_{\Lambda},\Phi_{f,g,h}\rangle,
			\qquad f,g,h\in\mathcal S(\R).
		\end{equation}
		Here \(\Phi_{f,g,h}\in\mathcal S(\R^2)\).  Taking \(\Lambda=\Omega_k\) and
		then letting \(k\to\infty\) identifies the operator-norm limit \(T\) with the
		radial principal-value operator \(T_\Omega\) on Schwartz inputs, and hence
		with its bounded extension.  This completes the proof.
	\end{proof}
	
	\subsection{Consequences of the rotational criterion}
	
	The profile characterization separates the even and odd parts of the angular
	kernel.  This leads to the following sufficient criterion, in which the
	logarithmic assumption is imposed only on the even component.
	
	\begin{corollary}
		\label{cor:parity-endpoint}
		Let
		\[
		1<p_1,p_2<\infty,
		\qquad
		1<p<\infty,
		\qquad
		\frac1p=\frac1{p_1}+\frac1{p_2}.
		\]
		Suppose that the mean-zero kernel
		\(\Omega=\Omega_e+\Omega_o\) satisfies
		\[
		\Omega_e\in L\log L(\Sph),
		\qquad
		\Omega_o\in L^1(\Sph),
		\]
		where \(\Omega_e\) and \(\Omega_o\) are respectively antipodally even
		and odd. Then
		\[
		\norm{T_\Omega(f,g)}_{L^p}
		\le
		C_{p_1,p_2}
		\left(
		\norm{\Omega_e}_{L\log L(\Sph)}
		+
		\norm{\Omega_o}_{L^1(\Sph)}
		\right)
		\norm f_{L^{p_1}}\norm g_{L^{p_2}}.
		\]
		In particular, if \(\Omega\in L^1(\Sph)\) is antipodally odd, then its
		mean is automatically zero and
		\[
		\norm{T_\Omega(f,g)}_{L^p}
		\le
		C_{p_1,p_2}
		\norm{\Omega}_{L^1(\Sph)}
		\norm f_{L^{p_1}}\norm g_{L^{p_2}}.
		\]
	\end{corollary}
	
	\begin{proof}
		The Riesz--Zygmund theorem
		\cite[Chapter~VII]{Zygmund1959}, together with the embedding
		\(L\log L(\Sph)\subset L^1(\Sph)\), gives
		\[
		\norm{\Omega_e}_{H^1(\Sph)}
		\le
		C\norm{\Omega_e}_{L\log L(\Sph)}.
		\]
		The first estimate now follows from
		Theorem~\ref{thm:rotational-main}. If \(\Omega\) is antipodally odd, then
		\(\Omega_e=0\) and \(\Omega_o=\Omega\). Rotational invariance of
		\(d\sigma\) also gives
		\[
		\int_{\Sph}\Omega(\theta)\,d\sigma(\theta)
		=
		\int_{\Sph}\Omega(-\theta)\,d\sigma(\theta)
		=
		-\int_{\Sph}\Omega(\theta)\,d\sigma(\theta),
		\]
		so the cancellation condition is automatic. The final estimate is the
		odd-kernel specialization of the first one.
	\end{proof}
	
	The next corollary recovers the one-dimensional \(L^q\) bound for every
	\(q>1\), throughout the finite-exponent Banach range covered by the
	rotational criterion.
	
	\begin{corollary}
		\label{cor:Lq-one-dimensional}
		Let \(q>1\), and let \(\Omega\in L^q(\Sph)\) have mean zero. Then, for
		every
		\[
		1<p_1,p_2<\infty,
		\qquad
		1<p<\infty,
		\qquad
		\frac1p=\frac1{p_1}+\frac1{p_2},
		\]
		the operator \(T_\Omega\) is bounded from
		\(L^{p_1}(\R)\times L^{p_2}(\R)\) to \(L^p(\R)\), and
		\[
		\norm{T_\Omega(f,g)}_{L^p}
		\le
		C_{p_1,p_2,q}
		\norm{\Omega}_{L^q(\Sph)}
		\norm f_{L^{p_1}}\norm g_{L^{p_2}}.
		\]
	\end{corollary}
	
	\begin{proof}
		Since \(\sigma(\Sph)=1\) and \(q>1\),
		\[
		L^q(\Sph)\hookrightarrow L\log L(\Sph)
		\hookrightarrow L^1(\Sph).
		\]
		The antipodal projections are bounded on all these spaces; hence
		\[
		\norm{\Omega_e}_{L\log L(\Sph)}
		+
		\norm{\Omega_o}_{L^1(\Sph)}
		\le
		C_q\norm{\Omega}_{L^q(\Sph)}.
		\]
		Corollary~\ref{cor:parity-endpoint} gives the asserted bound.
	\end{proof}
	
	\subsection{Proof of the \texorpdfstring{\(L\log L\)}{L log L} endpoint}
	
	We use the following Riesz--Zygmund theorem
	\cite[Chapter~VII]{Zygmund1959}.
	
	\begin{lemma}
		\label{lem:zygmund}
		For every \(F\in L\log L(\T)\), the periodic Hilbert transform satisfies
		\[
		\norm{H_\T F}_{L^1(\T)}
		\le
		C\norm F_{L\log L(\T)}.
		\]
	\end{lemma}
	\begin{lemma}\label{lem:log-orlicz}
		Let \(K_0(t)=-\log|\cos t|\).  For every \(F\in L\log L(\T)\),
		\[
		\sup_{s\in\T}\int_\T |F(t)|\bigl(1+K_0(t-s)\bigr)\,d\sigma(t)
		\le C\norm F_{L\log L(\T)}.
		\]
	\end{lemma}
	
	\begin{proof}
		The distribution function of \(K_0\) satisfies
		\(\sigma\{K_0>\lambda\}\le Ce^{-\lambda}\).  Hence \(1+K_0(\cdot-s)\) has
		uniformly bounded \(\exp L\) norm.  The Orlicz--H\"older inequality for the
		complementary pair \(L\log L\) and \(\exp L\) yields the estimate.
	\end{proof}
	
	The next proposition records the resulting quantitative regularity of the
	angular profile.  Its \(W^{1,1}\) conclusion is stronger than the \(BV\)
	regularity needed for the rotational estimate.
	
	\begin{proposition}\label{prop:profile-bv}
		Let \(\Omega\in L\log L(\Sph)\) have mean zero.  The profile
		\[
		\cM_\Omega(t)
		=
		c_{\rm rad}\int_\T
		\Omega(\theta(\varphi))
		\left[
		-\log|\cos(\varphi-t)|
		-\frac{i\pi}{2}\sgn(\cos(\varphi-t))
		\right]\,d\sigma(\varphi)
		\]
		has a \(W^{1,1}\) representative and satisfies
		\[
		\norm{\cM_\Omega}_{L^\infty}
		+
		\Var(\cM_\Omega)
		\le
		C\norm{\Omega}_{L\log L(\Sph)}.
		\]
	\end{proposition}
	
	\begin{proof}
		The estimate follows from the logarithmic integral bound and the
		Riesz--Zygmund theorem; related arguments appear in
		\cite{GrafakosStefanovSurvey1999,Stefanov2000}.  We include the short argument
		in the present normalization.  Put \(q(\varphi)=\Omega(\theta(\varphi))\) and
		decompose \(q=q_e+q_o\) with respect to
		the antipodal map.  Lemma \ref{lem:log-orlicz} controls the logarithmic part
		in \(L^\infty\), while the sign part is bounded by \(\norm{q_o}_{L^1}\).
		Lemma \ref{lem:zygmund} gives
		\[
		\norm{H_\T q_e}_{L^1(\T)}
		\le
		C\norm{q_e}_{L\log L(\T)}
		\le
		C\norm{\Omega}_{L\log L(\Sph)}.
		\]
		Lemma \ref{lem:profile-derivatives} then yields the derivative bound, and
		\(L\log L\hookrightarrow L^1\) controls the odd part.
	\end{proof}
	
	\begin{proof}[Proof of Theorem \ref{thm:LlogL-main}]
		The antipodal projections are bounded on the relevant Orlicz and Lebesgue
		spaces. Since \(\Sph\) has finite measure,
		\[
		\norm{\Omega_e}_{L\log L(\Sph)}
		+
		\norm{\Omega_o}_{L^1(\Sph)}
		\le
		C\norm{\Omega}_{L\log L(\Sph)}.
		\]
		The conclusion follows from Corollary~\ref{cor:parity-endpoint}.
	\end{proof}
	
	\subsection{Sharpness in the Orlicz scale}
	
	The sharpness example is built from two short arcs of the same size.  We place
	positive mass near a direction where the logarithmic part of the angular
	profile is large and the same amount of negative mass on a distant arc to
	enforce cancellation.  If the arcs have size \(e^{-N}\), the resulting profile
	grows like \(N\), whereas the \(L(\log L)^A\)-norm of the kernel grows only like
	\(N^A\).  A localized Fourier test transfers this difference to the operator
	norm and rules out every \(A<1\).
	
	Let \(0<\mu<1/100\), and let \(E,F\subset \Sph\) be disjoint with
	\(\sigma(E)=\sigma(F)=\mu\).  Define
	\[
	a_{E,F}=\mu^{-1}(\one_E-\one_F).
	\]
	
	\begin{lemma}\label{lem:orlicz-atom}
		Let \(A\ge0\), and let \(L(\log L)^A\) be defined by the Luxemburg norm
		associated with
		\[
		\Phi_A(t)=t\log^A(e+t),\qquad t\ge0.
		\]
		If \(N=\log(1/\mu)\), then
		\[
		\norm{a_{E,F}}_{L(\log L)^A(\Sph)}
		\simeq_A (1+N)^A.
		\]
	\end{lemma}
	
	\begin{proof}
		Since \(|a_{E,F}|=\mu^{-1}\) on \(E\cup F\) and
		\(\sigma(E\cup F)=2\mu\), for every \(\lambda>0\),
		\[
		\int_{\Sph}
		\Phi_A\left(\frac{|a_{E,F}|}{\lambda}\right)d\sigma
		=
		\frac2\lambda
		\log^A\left(e+\frac{\mu^{-1}}{\lambda}\right).
		\]
		Write \(\mu^{-1}=e^N\) and \(M_A=(1+N)^A\).  If
		\(\lambda=C_AM_A\) with \(C_A\) sufficiently large, then
		\[
		\log\left(e+\frac{e^N}{\lambda}\right)\le C(1+N),
		\]
		and hence
		\[
		\int_{\Sph}
		\Phi_A\left(\frac{|a_{E,F}|}{\lambda}\right)d\sigma
		\le
		\frac{C_A'}{C_A}\le1.
		\]
		Thus \(\norm{a_{E,F}}_{L(\log L)^A}\le C_AM_A\).
		
		Conversely, take \(\lambda=c_AM_A\), where \(c_A>0\) is sufficiently small.
		For all \(N\ge\log 100\),
		\[
		\log\left(e+\frac{e^N}{\lambda}\right)\ge c(1+N),
		\]
		after possibly decreasing \(c_A\) by a factor depending only on \(A\).
		Therefore
		\[
		\int_{\Sph}
		\Phi_A\left(\frac{|a_{E,F}|}{\lambda}\right)d\sigma
		\ge
		\frac{c_A'}{c_A}>1.
		\]
		By the definition of the Luxemburg norm,
		\(\norm{a_{E,F}}_{L(\log L)^A}\ge c_AM_A\).
	\end{proof}
	
	For all sufficiently large \(N\), choose
	\[
	E_N=\{\theta(\varphi):|\varphi-\pi/2|\le\pi e^{-N}\},
	\qquad
	F_N=\{\theta(\varphi):|\varphi|\le\pi e^{-N}\},
	\]
	and put
	\[
	\Omega_N=e^N(\one_{E_N}-\one_{F_N}).
	\]
	Then \(\Omega_N\) has mean zero and
	\[
	\norm{\Omega_N}_{L(\log L)^A}\simeq_A N^A.
	\]
	
	\begin{lemma}\label{lem:large-multiplier}
		There is a cone \(\Gamma_N\) of angular aperture comparable to \(e^{-N}\)
		around \(u_0=(1,0)\) such that
		\[
		\operatorname{Re}m_{\Omega_N}(\xi,\eta)\ge cN,
		\qquad
		(\xi,\eta)\in\Gamma_N,
		\]
		for all large \(N\).
	\end{lemma}
	
	\begin{proof}
		Fix \(c_0>0\) sufficiently small and set
		\[
		\Gamma_N
		=
		\left\{
		(\xi,\eta)\ne(0,0):
		\left|
		\frac{(\xi,\eta)}{(\xi^2+\eta^2)^{1/2}}-u_0
		\right|
		\le c_0e^{-N}
		\right\}.
		\]
		It suffices to bound
		\[
		A_N(u)=\int_{\Sph}\Omega_N(\theta)[-\log|\theta\cdot u|]\,d\sigma(\theta),
		\]
		because the real part of the profile is \(c_{\rm rad}A_N(u)\) with
		\(c_{\rm rad}>0\).  If \((\xi,\eta)\in\Gamma_N\) and
		\(u=(\xi,\eta)/(\xi^2+\eta^2)^{1/2}\), then \(|u-u_0|\le c_0e^{-N}\).  For
		\(\theta\in E_N\),
		\[
		|\theta\cdot u|\le C e^{-N},
		\qquad
		-\log|\theta\cdot u|\ge N-C.
		\]
		On \(F_N\), \(|\theta\cdot u|\ge1/4\) for all large \(N\), so the logarithmic
		term is bounded above by an absolute constant.  Since
		\(\sigma(E_N)=\sigma(F_N)=e^{-N}\) and the atom has height \(e^N\), we
		obtain
		\[
		A_N(u)
		\ge
		e^N e^{-N}(N-C)-Ce^N e^{-N}
		\ge cN
		\]
		for all sufficiently large \(N\).
	\end{proof}
	
	\begin{lemma}\label{lem:testing}
		Let \(1<p_1,p_2<\infty\), \(1<p<\infty\), and
		\(1/p=1/p_1+1/p_2\).  Let \(m\in L^\infty(\R^2)\) be measurable, and let
		\(T_m\) be the associated bilinear Fourier multiplier, initially defined on
		Schwartz inputs.  Let \(L>0\), and let
		\[
		I=[\xi_0-\delta,\xi_0+\delta],
		\qquad
		J=[\eta_0-\delta,\eta_0+\delta],
		\qquad
		\delta>0.
		\]
		Suppose, almost everywhere on \(I\times J\), that
		\[
		\operatorname{Re}m(\xi,\eta)\ge L.
		\]
		Then
		\[
		\norm{T_m}_{L^{p_1}\times L^{p_2}\to L^p}
		\ge
		c_{p_1,p_2}L,
		\]
		where \(c_{p_1,p_2}>0\) is independent of \(m,L,\xi_0,\eta_0\), and
		\(\delta\).
	\end{lemma}
	
	\begin{proof}
		If the operator norm is infinite, there is nothing to prove.  Otherwise, the
		associated trilinear form is bounded on Schwartz inputs by duality.  Choose
		Schwartz functions \(\phi,\psi,\chi\) such that
		\(\widehat\phi,\widehat\psi\ge0\), neither is identically zero,
		\[
		\operatorname{supp}\widehat\phi,
		\operatorname{supp}\widehat\psi
		\subset[-1/4,1/4],
		\qquad
		\widehat\chi=1
		\quad\text{on }[-1/2,1/2].
		\]
		We may in addition take \(\widehat\chi\) to be real-valued on
		\([-1/2,1/2]\).
		Set
		\[
		f(x)=\delta^{1/p_1}e^{2\pi i\xi_0x}\phi(\delta x),
		\qquad
		g(x)=\delta^{1/p_2}e^{2\pi i\eta_0x}\psi(\delta x),
		\]
		and
		\[
		h(x)=\delta^{1/p'}e^{2\pi i(\xi_0+\eta_0)x}\chi(\delta x),
		\]
		where \(p'\) is conjugate to \(p\).  Then
		\[
		\norm f_{p_1}=\norm\phi_{p_1},
		\qquad
		\norm g_{p_2}=\norm\psi_{p_2},
		\qquad
		\norm h_{p'}=\norm\chi_{p'}.
		\]
		Moreover,
		\[
		\widehat f(\xi)
		=
		\delta^{1/p_1-1}
		\widehat\phi\left(\frac{\xi-\xi_0}{\delta}\right),
		\qquad
		\widehat g(\eta)
		=
		\delta^{1/p_2-1}
		\widehat\psi\left(\frac{\eta-\eta_0}{\delta}\right),
		\]
		and
		\[
		\overline{\widehat h(\xi+\eta)}
		=
		\delta^{1/p'-1}
		\overline{\widehat\chi
			\left(\frac{\xi+\eta-\xi_0-\eta_0}{\delta}\right)} .
		\]
		On the support of \(\widehat f(\xi)\widehat g(\eta)\), we have
		\((\xi,\eta)\in I\times J\) and
		\[
		\frac{\xi+\eta-\xi_0-\eta_0}{\delta}\in[-1/2,1/2],
		\]
		so \(\overline{\widehat\chi}=1\) there.  The trilinear form satisfies
		\[
		\Lambda_m(f,g,h)
		=
		\iint m(\xi,\eta)
		\widehat f(\xi)\widehat g(\eta)
		\overline{\widehat h(\xi+\eta)}\,d\xi d\eta .
		\]
		After making the change of variables
		\[
		\xi=\xi_0+\delta a,
		\qquad
		\eta=\eta_0+\delta b,
		\]
		the power of \(\delta\) is
		\[
		\delta^{1/p_1-1}
		\delta^{1/p_2-1}
		\delta^{1/p'-1}
		\delta^2
		=
		\delta^{1/p+1/p'-1}
		=
		1.
		\]
		Therefore
		\[
		\begin{aligned}
			\operatorname{Re}\Lambda_m(f,g,h)
			&\ge
			L
			\iint
			\widehat\phi(a)\widehat\psi(b)\widehat\chi(a+b)\,da\,db  \\
			&=
			L\,C_{\phi,\psi,\chi},
		\end{aligned}
		\]
		where
		\[
		C_{\phi,\psi,\chi}
		=
		\iint
		\widehat\phi(a)\widehat\psi(b)\widehat\chi(a+b)\,da\,db
		>0.
		\]
		On the other hand,
		\[
		|\Lambda_m(f,g,h)|
		\le
		\norm{T_m}_{L^{p_1}\times L^{p_2}\to L^p}
		\norm f_{p_1}\norm g_{p_2}\norm h_{p'}.
		\]
		Combining the two estimates yields the desired lower bound, with
		\[
		c_{p_1,p_2}
		=
		\frac{C_{\phi,\psi,\chi}}
		{\norm\phi_{p_1}\norm\psi_{p_2}\norm\chi_{p'}}.
		\qedhere
		\]
	\end{proof}
	
	\begin{proposition}\label{prop:operator-lower}
		Assume
		\[
		1<p_1,p_2<\infty,\qquad
		1<p<\infty,\qquad
		\frac1p=\frac1{p_1}+\frac1{p_2}.
		\]
		For the angular atoms
		\[
		\Omega_N=e^N(\one_{E_N}-\one_{F_N})
		\]
		defined above, there exist constants \(c=c(p_1,p_2)>0\) and \(N_0\) such
		that, for all \(N\ge N_0\),
		\[
		\norm{T_{\Omega_N}}_{L^{p_1}\times L^{p_2}\to L^p}
		\ge cN .
		\]
	\end{proposition}
	
	\begin{proof}
		Let \(c_0>0\) be the constant in the definition of \(\Gamma_N\) from Lemma
		\ref{lem:large-multiplier}, and choose \(0<c_{\rm rec}<c_0/4\).  For all
		sufficiently large \(N\), the rectangle
		\[
		\xi\in[1-c_{\rm rec}e^{-N},1+c_{\rm rec}e^{-N}],
		\qquad
		\eta\in[-c_{\rm rec}e^{-N},c_{\rm rec}e^{-N}]
		\]
		is contained in \(\Gamma_N\), because the distance between the normalized
		direction and \(u_0\) is at most \(3c_{\rm rec}e^{-N}<c_0e^{-N}\).  On this
		rectangle,
		\[
		\operatorname{Re}m_{\Omega_N}(\xi,\eta)\ge c_1N .
		\]
		Lemma \ref{lem:testing}, with \(L=c_1N\), gives the asserted lower bound.
	\end{proof}
	
	\begin{proof}[Proof of Theorem \ref{thm:Orlicz-sharp}]
		Theorem \ref{thm:LlogL-main} gives \(A_*(p_1,p_2)\le1\).  Conversely, if an
		\(L(\log L)^A\) estimate held uniformly for all mean-zero
		\(\Omega\in L(\log L)^A(\Sph)\), then Proposition
		\ref{prop:operator-lower} and Lemma \ref{lem:orlicz-atom} would give
		\[
		N
		\lesssim
		\norm{T_{\Omega_N}}
		\lesssim_A
		\norm{\Omega_N}_{L(\log L)^A}
		\lesssim_A N^A .
		\]
		Letting \(N\to\infty\) rules out every \(A<1\).  Hence
		\(A_*(p_1,p_2)\ge1\), and the two inequalities give
		\(A_*(p_1,p_2)=1\).
	\end{proof}
	
	\vspace{0.3cm}

\section{Boundedness under the \(\cK_{1/2,\beta}\) kernel condition}
	
	\subsection{Dyadic decomposition and preliminary estimates}
	
	Let \(\chi\in C_c^\infty((1/2,2))\) satisfy
	\[
	\sum_{i\in\mathbb Z}\chi(2^{-i}r)=1,\qquad r>0.
	\]
	Set \(\chi_i(x)=\chi(2^{-i}|x|)\), and let \(K\) denote the homogeneous kernel
	in \eqref{eq:T-def}:
	\[
	K(y_1,y_2)=\frac{\Omega((y_1,y_2)/\abs{(y_1,y_2)})}{\abs{(y_1,y_2)}^2},
	\]
	and define
	\[
	K^i:=\chi_iK,
	\qquad
	K^0:=\chi_0K.
	\]
	Choose also a smooth Littlewood--Paley cutoff
	\(\eta\in C_c^\infty((1/2,2))\) such that
	\[
	\sum_{j\in\mathbb Z}\eta(2^{-j}r)=1
	\qquad (r>0),
	\]
	and let \(\Delta_j\) denote the Fourier multiplier with symbol
	\(\eta(2^{-j}\abs{\xi})\).  We then set
	\[
	K_j^i:=\Delta_{j-i}K^i,
	\qquad
	K_j:=\sum_{i\in\mathbb Z}K_j^i,
	\qquad
	T_j:=T_{K_j}.
	\]
	The decomposition is chosen so that
	\[
	T_\Omega=\sum_{j\in\mathbb Z}T_j
	\]
	in \(\mathcal S'\). Since the identity involves two infinite indices and
	the original operator is defined by radial principal value, its precise
	reconstruction and cutoff independence are recorded in
	Lemma~\ref{lem:dyadic-reconstruction} below.
	
	The homogeneity of \(K\) implies
	\[
	K^i(x)=2^{-2i}K^0(2^{-i}x).
	\]
	Consequently, \(\widehat{K_j}\) is a sum of dyadic dilates of the prototype multiplier \(\widehat{K_j^0}\).
	
	The defining directional norm controls the \(L^1\) norm of \(\Omega\).
	
	\begin{lemma}\label{lem:L1-from-K}
		For every \(\beta>0\), there exists \(C_\beta>0\) such that
		\[
		\norm{\Omega}_{L^1(\Sph)}\le C_\beta\norm{\Omega}_{\cK_{1/2,\beta}}.
		\]
	\end{lemma}
	
	\begin{proof}
		For \(\theta\in\Sph\), set
		\[
		c_\beta(\theta):=
		\int_{\Sph}
		\frac{1}{\abs{\theta\cdot\xi}^{1/2}}
		\left(\log\frac1{\abs{\theta\cdot\xi}}\right)^\beta
		\,d\sigma(\xi).
		\]
		By rotational invariance, \(c_\beta(\theta)\) is independent of \(\theta\);
		denote its common value by \(c_\beta\).  Moreover, \(0<c_\beta<\infty\): near
		either zero of \(\xi\mapsto\theta\cdot\xi\), the integrand is comparable to
		\[
		u^{-1/2}\left(\log\frac1u\right)^\beta,
		\]
		which is integrable near \(u=0\).  Tonelli's theorem and
		\(\sigma(\Sph)=1\) therefore give
		\begin{align*}
			c_\beta\norm{\Omega}_{L^1(\Sph)}
			&=
			\int_{\Sph}\int_{\Sph}
			\frac{\abs{\Omega(\theta)}}{\abs{\theta\cdot\xi}^{1/2}}
			\left(\log\frac1{\abs{\theta\cdot\xi}}\right)^\beta
			\,d\sigma(\theta)d\sigma(\xi)\\
			&\le
			\norm{\Omega}_{\cK_{1/2,\beta}}.
		\end{align*}
		Dividing by \(c_\beta\) proves the assertion.
	\end{proof}
	
	\subsection{Endpoint Fourier decay}
	
	The logarithmic gain enters through the following lemma.
	
	\begin{lemma}\label{lem:prototype-fourier}
		Let \(\Omega\in\cK_{1/2,\beta}\) have mean zero. Then the Fourier
		transform of \(K^0\) satisfies the following estimates.
		\begin{enumerate}
			\item for \(\abs{\xi}\le 1\),
			\begin{equation}\label{eq:K0-low}
				\abs{\widehat{K^0}(\xi)}\lesssim \norm{\Omega}_{\cK_{1/2,\beta}}\,\abs{\xi};
			\end{equation}
			\item for \(\abs{\xi}\ge 4\),
			\begin{equation}\label{eq:K0-high}
				\abs{\widehat{K^0}(\xi)}\lesssim \norm{\Omega}_{\cK_{1/2,\beta}}\,\abs{\xi}^{-1/2}(\log\abs{\xi})^{-\beta}.
			\end{equation}
		\end{enumerate}
		Moreover, for every multi-index \(\alpha\neq 0\) and \(\abs{\xi}\ge 4\) one has
		\begin{equation}\label{eq:K0-deriv-high}
			\abs{\p^\alpha\widehat{K^0}(\xi)}
			\lesssim
			\norm{\Omega}_{\cK_{1/2,\beta}}\,\abs{\xi}^{-1/2}(\log\abs{\xi})^{-\beta}.
		\end{equation}
	\end{lemma}
	
	The implicit constants may depend on \(\beta\) and the fixed cutoff
	\(\chi\); in the derivative estimate, they may also depend on \(\alpha\).
	They are independent of \(\xi\) and \(\Omega\).
	
	\begin{remark}
		The constants \(1\) and \(4\) merely separate the low- and high-frequency
		regimes.
		On the intermediate region \(1<|\xi|<4\), the trivial estimate
		\[
		|\widehat{K^0}(\xi)|\le \|K^0\|_{L^1(\mathbb R^2)}
		\lesssim \|\Omega\|_{\cK_{1/2,\beta}}.
		\]
		Thus this compact region can be absorbed into the implicit constant.
	\end{remark}
	
	\begin{proof}
		\emph{Setup.}
		Writing \(x=r\theta\), with \(r>0\) and \(\theta\in\Sph\), and recalling that
		\(d\sigma=d\varphi/(2\pi)\), so that
		\[
		dx=r\,dr\,d\varphi=2\pi r\,dr\,d\sigma(\theta),
		\]
		we obtain
		\begin{equation}\label{eq:polar-K0}
			\widehat{K^0}(\xi)
			=
			2\pi\int_{\Sph}\Omega(\theta)
			\Bigl(\int_{1/2}^{2}\chi(r)e^{-2\pi i r\,\xi\cdot\theta}\frac{dr}{r}\Bigr)
			d\sigma(\theta).
		\end{equation}
		Set
		\[
		I(t):=2\pi\int_{1/2}^{2}\chi(r)e^{-2\pi i rt}\frac{dr}{r}.
		\]
		Since \(r\mapsto \chi(r)/r\) is a smooth compactly supported function on \((1/2,2)\), one integration by parts gives
		\begin{equation}\label{eq:I-basic}
			\abs{I(t)}\lesssim \min\{1,\abs{t}^{-1}\}
			\qquad (t\in\R).
		\end{equation}
		
		\smallskip
		\noindent\emph{Low frequencies.}
		We first prove \eqref{eq:K0-low}. Differentiating under the integral sign gives
		\[
		I'(t)=-4\pi^2 i\int_{1/2}^{2}\chi(r)e^{-2\pi i rt}\,dr,
		\]
		so \(I'\) is bounded and \(I\) is Lipschitz. In particular,
		\begin{equation}\label{eq:I-lipschitz}
			\abs{I(t)-I(0)}\le \norm{I'}_{L^\infty(\R)}\abs{t}\lesssim \abs{t}
			\qquad (t\in\R).
		\end{equation}
		Since \(\Omega\) has mean zero, \eqref{eq:polar-K0} gives
		\[
		\widehat{K^0}(\xi)=\int_{\Sph}\Omega(\theta)\bigl(I(\xi\cdot\theta)-I(0)\bigr)\,d\sigma(\theta).
		\]
		Combining \eqref{eq:I-lipschitz} with Lemma~\ref{lem:L1-from-K} gives, for
		\(\abs{\xi}\le1\),
		\begin{align*}
			\abs{\widehat{K^0}(\xi)}
			&\le
			\int_{\Sph}\abs{\Omega(\theta)}\abs{I(\xi\cdot\theta)-I(0)}\,d\sigma(\theta)
			\\
			&\lesssim
			\abs{\xi}\int_{\Sph}\abs{\Omega(\theta)}\,d\sigma(\theta)
			\lesssim
			\abs{\xi}\,\norm{\Omega}_{\cK_{1/2,\beta}}.
		\end{align*}
		This proves \eqref{eq:K0-low}.
		
		\smallskip
		\noindent\emph{High frequencies.}
		We now prove \eqref{eq:K0-high}. Let \(\rho:=\abs{\xi}\ge4\) and write \(\omega:=\xi/\abs{\xi}\in\Sph\). For \(\theta\in\Sph\) set
		\[
		u:=\abs{\theta\cdot\omega}\in[0,1].
		\]
		By \eqref{eq:I-basic},
		\[
		\abs{I(\rho u)}\lesssim \min\Bigl\{1,\frac1{\rho u}\Bigr\}.
		\]
		We therefore analyze the scalar function
		\begin{equation}\label{eq:def-f}
			f(u):=u^{1/2}\Bigl(\log\frac1u\Bigr)^{-\beta}
			\min\Bigl(1,\frac1{\rho u}\Bigr),
			\qquad 0<u\le e^{-1}.
		\end{equation}
		We claim that
		\begin{equation}\label{eq:f-max}
			\sup_{0<u\le e^{-1}}f(u)
			\lesssim_\beta
			\rho^{-1/2}(\log\rho)^{-\beta}.
		\end{equation}
		We split the analysis at \(u=\rho^{-1}\), where the two terms in
		\(\min\{1,(\rho u)^{-1}\}\) change.
		
		On \(0<u\le\rho^{-1}\) we have
		\[
		f(u)=f_1(u):=u^{1/2}\Bigl(\log\frac1u\Bigr)^{-\beta}.
		\]
		Since \(\rho\ge4>e\), the interval \((0,\rho^{-1}]\) is contained in \((0,e^{-1}]\). A direct differentiation yields
		\[
		f_1'(u)
		=
		\frac{u^{-1/2}}{\bigl(\log(1/u)\bigr)^{\beta+1}}
		\Bigl(\frac12\log\frac1u+\beta\Bigr)>0,
		\]
		so \(f_1\) is increasing. Hence
		\begin{equation}\label{eq:f1-max}
			\sup_{0<u\le\rho^{-1}}f(u)=f_1(\rho^{-1})=\rho^{-1/2}(\log\rho)^{-\beta}.
		\end{equation}
		
		On the complementary interval \(\rho^{-1}\le u\le e^{-1}\) we have
		\[
		f(u)=f_2(u):=\rho^{-1}u^{-1/2}\Bigl(\log\frac1u\Bigr)^{-\beta}.
		\]
		Differentiating again gives
		\[
		f_2'(u)
		=
		\rho^{-1}
		\frac{u^{-3/2}}{\bigl(\log(1/u)\bigr)^{\beta+1}}
		\Bigl(\beta-\frac12\log\frac1u\Bigr).
		\]
		Thus the only possible critical point is at \(u=e^{-2\beta}\).  The derivative
		is negative to the left of this point and positive to its right, so this
		critical point, when it belongs to the interval, is a minimum.  Consequently,
		the maximum of \(f_2\) on \([\rho^{-1},e^{-1}]\) is attained at an endpoint.
		The endpoint values are
		\[
		f_2(\rho^{-1})=\rho^{-1/2}(\log\rho)^{-\beta},
		\qquad
		f_2(e^{-1})=e^{1/2}\rho^{-1},
		\]
		and
		\[
		\rho^{-1}
		\le C_\beta\rho^{-1/2}(\log\rho)^{-\beta},
		\qquad \rho\ge4.
		\]
		Thus both endpoint values are bounded by a constant multiple of
		\(\rho^{-1/2}(\log\rho)^{-\beta}\). Hence
		\begin{equation}\label{eq:f2-max}
			\sup_{\rho^{-1}\le u\le e^{-1}}f_2(u)
			\lesssim_\beta
			\rho^{-1/2}(\log\rho)^{-\beta}.
		\end{equation}
		Combining \eqref{eq:f1-max} and \eqref{eq:f2-max} proves \eqref{eq:f-max}.
		
		We now return to \eqref{eq:polar-K0}. The decomposition of the sphere into
		\[
		E_1:=\{\theta\in\Sph: \abs{\theta\cdot\omega}\le e^{-1}\},
		\qquad
		E_2:=\{\theta\in\Sph: \abs{\theta\cdot\omega}> e^{-1}\},
		\]
		separates the part where the logarithmic weight in \(\cK_{1/2,\beta}\) can be
		used from the part where we instead use the oscillatory decay
		\(\abs{I(\rho u)}\lesssim(\rho u)^{-1}\).
		
		On \(E_1\), writing \(u=\abs{\theta\cdot\omega}\) and using \eqref{eq:f-max}, we have
		\[
		\min\Bigl\{1,\frac1{\rho u}\Bigr\}
		\lesssim
		\rho^{-1/2}(\log\rho)^{-\beta}
		u^{-1/2}\Bigl(\log\frac1u\Bigr)^\beta.
		\]
		Hence
		\begin{align*}
			\int_{E_1}\abs{\Omega(\theta)}\abs{I(\rho u)}\,d\sigma(\theta)
			&\lesssim
			\rho^{-1/2}(\log\rho)^{-\beta}
			\int_{\Sph}
			\frac{\abs{\Omega(\theta)}}{\abs{\theta\cdot\omega}^{1/2}}
			\Bigl(\log\frac1{\abs{\theta\cdot\omega}}\Bigr)^\beta\,d\sigma(\theta)
			\\
			&\le
			\rho^{-1/2}(\log\rho)^{-\beta}\norm{\Omega}_{\cK_{1/2,\beta}}.
		\end{align*}
		
		On \(E_2\) we have \(u>e^{-1}\), so \eqref{eq:I-basic} yields
		\[
		\abs{I(\rho u)}\lesssim (\rho u)^{-1}\lesssim \rho^{-1}.
		\]
		Therefore, by Lemma~\ref{lem:L1-from-K},
		\begin{align*}
			\int_{E_2}\abs{\Omega(\theta)}\abs{I(\rho u)}\,d\sigma(\theta)
			&\lesssim
			\rho^{-1}\norm{\Omega}_{L^1(\Sph)}
			\\
			&\lesssim
			\rho^{-1}\norm{\Omega}_{\cK_{1/2,\beta}}
			\lesssim_\beta
			\rho^{-1/2}(\log\rho)^{-\beta}\norm{\Omega}_{\cK_{1/2,\beta}}.
		\end{align*}
		Combining the estimates on \(E_1\) and \(E_2\) proves \eqref{eq:K0-high}.
		
		\smallskip
		\noindent\emph{Derivative bounds.}
		Let \(\alpha\neq 0\). Since \(K^0\in L^1(\R^2)\) is compactly supported, differentiation under the Fourier transform gives
		\[
		\partial^\alpha\widehat{K^0}(\xi)
		=
		\widehat{(-2\pi i x)^\alpha K^0}(\xi).
		\]
		Writing \(x=r\theta\) with \(r\in(1/2,2)\) and \(\theta\in\Sph\), and again
		using \(dx=2\pi r\,dr\,d\sigma(\theta)\), together with
		\[
		K^0(x)=\chi(r)\frac{\Omega(\theta)}{r^2},
		\]
		we obtain
		\begin{align*}
			\partial^\alpha\widehat{K^0}(\xi)
			&=
			2\pi(-2\pi i)^{|\alpha|}
			\int_{\Sph}\int_{1/2}^2
			(r\theta)^\alpha
			\chi(r)\frac{\Omega(\theta)}{r^2}
			e^{-2\pi i r\,\xi\cdot\theta}\, r\,dr\,d\sigma(\theta) \\
			&=
			\int_{\Sph}\theta^\alpha \Omega(\theta)\,
			I_\alpha(\xi\cdot\theta)\,d\sigma(\theta),
		\end{align*}
		where
		\[
		I_\alpha(t):=
		2\pi(-2\pi i)^{|\alpha|}
		\int_{1/2}^2 r^{|\alpha|-1}\chi(r)e^{-2\pi i rt}\,dr.
		\]
		Since \(r\mapsto r^{|\alpha|-1}\chi(r)\) is smooth and compactly supported in \((1/2,2)\), one integration by parts yields
		\[
		|I_\alpha(t)|\lesssim_\alpha \min\{1,|t|^{-1}\}.
		\]
		Moreover,
		\[
		|\theta^\alpha \Omega(\theta)|\le |\Omega(\theta)|,
		\]
		and hence
		\[
		\|\theta^\alpha\Omega\|_{\cK_{1/2,\beta}}
		\le
		\|\Omega\|_{\cK_{1/2,\beta}}.
		\]
		Therefore the same high-frequency argument used above for \(\widehat{K^0}\), with
		\(\Omega\) replaced by \(\theta^\alpha\Omega\) and \(I\) replaced by \(I_\alpha\),
		gives
		\[
		|\partial^\alpha\widehat{K^0}(\xi)|
		\lesssim_\alpha
		\|\Omega\|_{\cK_{1/2,\beta}}
		|\xi|^{-1/2}(\log |\xi|)^{-\beta},
		\qquad |\xi|\ge 4.
		\]
		This proves \eqref{eq:K0-deriv-high}.
	\end{proof}
	
	\begin{corollary}\label{cor:Kj0-high}
		Let \(j\ge2\). Then on the support of \(\eta(2^{-j}\abs{\xi})\) one has
		\begin{equation}\label{eq:Kj0-pointwise}
			\abs{\widehat{K_j^0}(\xi)}
			\lesssim
			\norm{\Omega}_{\cK_{1/2,\beta}}\,2^{-j/2}j^{-\beta},
		\end{equation}
		and, for each multi-index \(\alpha\neq0\),
		\begin{equation}\label{eq:Kj0-deriv}
			\abs{\p^\alpha \widehat{K_j^0}(\xi)}
			\lesssim
			\norm{\Omega}_{\cK_{1/2,\beta}}\,2^{-j/2}j^{-\beta}.
		\end{equation}
		The implicit constants may depend on \(\beta\); in the derivative estimate,
		they may also depend on \(\alpha\) and finitely many derivatives of \(\eta\).
		They are independent of \(j\) and \(\Omega\).
	\end{corollary}
	
	\begin{proof}
		Since
		\[
		\widehat{K_j^0}(\xi)=\eta(2^{-j}\abs{\xi})\widehat{K^0}(\xi),
		\]
		and
		\[
		\supp\eta(2^{-j}\abs{\xi})\subset\{2^{j-1}\le\abs{\xi}\le2^{j+1}\},
		\]
		we have \(\abs{\xi}\simeq 2^j\) on the support. Thus
		\[
		(\log\abs{\xi})^{-\beta}\simeq j^{-\beta}
		\qquad (j\ge2),
		\]
		because \(\log\abs{\xi}=j\log 2+O(1)\) there.
		
		If \(j\ge3\), then \(\abs{\xi}\ge4\) on the support, so
		\eqref{eq:Kj0-pointwise} follows from \eqref{eq:K0-high}. For \(j=2\), the
		support lies in the fixed annulus \(\{2\le \abs{\xi}\le 8\}\), where
		\[
		|\widehat{K^0}(\xi)|
		\le \|K^0\|_{L^1(\R^2)}
		\lesssim \|\Omega\|_{\cK_{1/2,\beta}}.
		\]
		Since \(2^{-j/2}j^{-\beta}\) is then a positive constant depending only on
		\(\beta\), \eqref{eq:Kj0-pointwise} follows for every \(j\ge2\).
		
		For \eqref{eq:Kj0-deriv}, Leibniz' rule gives
		\[
		\p^\alpha\widehat{K_j^0}(\xi)
		=
		\sum_{\alpha_1+\alpha_2=\alpha}
		C_{\alpha_1,\alpha_2}\,
		\p^{\alpha_1}\bigl[\eta(2^{-j}\abs{\xi})\bigr]\,
		\p^{\alpha_2}\widehat{K^0}(\xi).
		\]
		On the support of \(\eta(2^{-j}|\xi|)\), where \(|\xi|\simeq 2^j\), every
		derivative of the radial cutoff satisfies
		\[
		\abs{\p^\gamma[\eta(2^{-j}\abs{\xi})]}
		\le C_\gamma 2^{-j\abs{\gamma}}
		\qquad \text{on }\supp\eta(2^{-j}\abs{\xi}).
		\]
		For \(j\ge3\), the terms in the Leibniz expansion with \(\alpha_2=0\) are
		controlled by \eqref{eq:K0-high}, and those with \(\alpha_2\ne0\) by
		\eqref{eq:K0-deriv-high}. Thus every term is bounded by
		\[
		\lesssim
		\|\Omega\|_{\cK_{1/2,\beta}}\,2^{-j|\alpha_1|}\,2^{-j/2}j^{-\beta}
		\lesssim
		\|\Omega\|_{\cK_{1/2,\beta}}\,2^{-j/2}j^{-\beta}.
		\]
		If \(j=2\), then \(\supp\eta(2^{-j}|\xi|)\subset\{2\le |\xi|\le 8\}\), and for
		every multi-index \(\alpha_2\),
		\[
		|\partial^{\alpha_2}\widehat{K^0}(\xi)|
		\le \|(-2\pi i x)^{\alpha_2}K^0\|_{L^1(\R^2)}
		\lesssim_{\alpha_2}\|\Omega\|_{\cK_{1/2,\beta}}.
		\]
		Hence \eqref{eq:Kj0-deriv} also follows, because \(j=2\) is a fixed scale.
	\end{proof}

	\subsection[Low-frequency analysis for \(j\le0\)]
	{Low-frequency analysis for \texorpdfstring{\(j\le0\)}{j <= 0}}

	The negative scales are treated as in \cite{DQX0404}; we include the argument
	for completeness.
	
	\begin{proposition}\label{prop:negative-scales}
		Let \(1<p_1,p_2,p<\infty\) with \(1/p=1/p_1+1/p_2\). Then, for every \(j\le0\),
		\[
		\norm{T_j(f_1,f_2)}_{L^p(\R)}
		\lesssim
		2^j\norm{\Omega}_{\cK_{1/2,\beta}}
		\norm{f_1}_{L^{p_1}(\R)}\norm{f_2}_{L^{p_2}(\R)}.
		\]
		Consequently,
		\[
		\sum_{j\le0}\norm{T_j}_{L^{p_1}\times L^{p_2}\to L^p}
		\lesssim
		\norm{\Omega}_{\cK_{1/2,\beta}}.
		\]
	\end{proposition}
	
	\begin{proof}
		Set \(m_j:=\widehat{K_j}\). By homogeneity,
		\[
		m_j(\xi)=\sum_{i\in\mathbb Z}\eta(2^{i-j}\abs{\xi})\widehat{K^0}(2^i\xi).
		\]
		Fix \(\xi\neq0\). Only the indices \(i\) for which
		\(2^i\abs{\xi}\simeq2^j\) contribute. If \(j\le-1\), the support condition
		\(\supp\eta\subset(1/2,2)\) gives \(|2^i\xi|\le1\), and
		\eqref{eq:K0-low} yields
		\[
		\abs{\widehat{K^0}(2^i\xi)}
		\lesssim_\beta 2^j\norm{\Omega}_{\cK_{1/2,\beta}}.
		\]
		At the fixed scale \(j=0\), the low-frequency estimate need not hold throughout
		the support, where \(|2^i\xi|\) may be as large as \(2\). The trivial estimate
		\[
		\abs{\widehat{K^0}(2^i\xi)}
		\le \norm{K^0}_{L^1(\R^2)}
		\lesssim_\beta \norm{\Omega}_{\cK_{1/2,\beta}}
		=2^0\norm{\Omega}_{\cK_{1/2,\beta}}
		\]
		gives the required bound. Thus the zero-order Coifman--Meyer estimate holds
		for every \(j\le0\).
		Moreover, because \((-2\pi ix)^\alpha K^0\in L^1(\R^2)\) for every \(\alpha\),
		every derivative of \(\widehat{K^0}\) is bounded on bounded sets. In particular,
		for \(|\zeta|\le2\) and every multi-index \(\alpha\),
		\[
		\abs{\p^\alpha\widehat{K^0}(\zeta)}\lesssim_\alpha \norm{\Omega}_{L^1(\Sph)}
		\lesssim_\beta \norm{\Omega}_{\cK_{1/2,\beta}}
		\]
		by Lemma~\ref{lem:L1-from-K}. Let \(N_{\rm CM}\) be a differentiability order
		sufficient for the bilinear Coifman--Meyer theorem in dimension one, and fix a
		multi-index \(\alpha\) with \(1\le |\alpha|\le N_{\rm CM}\). For each
		contributing index \(i\), write
		\[
		s_i(\xi):=\eta(2^{i-j}|\xi|)\widehat{K^0}(2^i\xi).
		\]
		By Leibniz' rule,
		\[
		\partial^\alpha s_i(\xi)
		=
		\sum_{\alpha_1+\alpha_2=\alpha}
		C_{\alpha_1,\alpha_2}\,
		\partial^{\alpha_1}\bigl[\eta(2^{i-j}|\xi|)\bigr]\,
		\partial^{\alpha_2}\bigl[\widehat{K^0}(2^i\xi)\bigr].
		\]
		On the support of \(\eta(2^{i-j}|\xi|)\) one has \(2^i|\xi|\simeq 2^j\), hence
		\[
		2^{i-j}\simeq |\xi|^{-1}.
		\]
		Therefore
		\[
		\bigl|\partial^{\alpha_1}[\eta(2^{i-j}|\xi|)]\bigr|
		\lesssim |\xi|^{-|\alpha_1|}.
		\]
		If \(\alpha_2=0\), the zero-order estimate above (using
		\eqref{eq:K0-low} for \(j\le-1\) and the fixed-scale bound for \(j=0\)) gives
		\[
		|\widehat{K^0}(2^i\xi)|
		\lesssim_\beta 2^j\|\Omega\|_{\cK_{1/2,\beta}}.
		\]
		If \(\alpha_2\ne0\), the boundedness of the corresponding derivatives of
		\(\widehat{K^0}\) on bounded sets gives
		\[
		\bigl|\partial^{\alpha_2}[\widehat{K^0}(2^i\xi)]\bigr|
		\lesssim_\beta
		(2^i)^{|\alpha_2|}\|\Omega\|_{\cK_{1/2,\beta}}
		\lesssim
		2^{j|\alpha_2|}|\xi|^{-|\alpha_2|}\|\Omega\|_{\cK_{1/2,\beta}}
		\lesssim
		2^j|\xi|^{-|\alpha_2|}\|\Omega\|_{\cK_{1/2,\beta}},
		\]
		since \(j\le0\) and \(|\alpha_2|\ge1\). Combining these bounds yields
		\[
		|\partial^\alpha s_i(\xi)|
		\lesssim
		2^j\|\Omega\|_{\cK_{1/2,\beta}}|\xi|^{-|\alpha|}.
		\]
		Since only \(O(1)\) indices \(i\) contribute for each fixed \(\xi\ne0\), we conclude that
		\[
		|\partial^\alpha m_j(\xi)|
		\lesssim
		2^j\|\Omega\|_{\cK_{1/2,\beta}}|\xi|^{-|\alpha|}.
		\]
		Thus \(m_j\) satisfies the bilinear Coifman--Meyer estimates with constant
		\(\lesssim 2^j\norm{\Omega}_{\cK_{1/2,\beta}}\). The asserted operator bound
		follows from the bilinear Coifman--Meyer theorem; see
		\cite{GrafakosTorres2002}. The geometric sum over \(j\le0\) completes the proof.
	\end{proof}
	
	\begin{lemma}[Dyadic reconstruction and principal-value identification]
		\label{lem:dyadic-reconstruction}
		Let \(\Omega\in\cK_{1/2,\beta}(\Sph)\) have mean zero, and let
		\(K^i,K_j^i,K_j\), and \(T_j\) be defined as above. Then
		\[
		\sum_{j\in\mathbb Z}K_j=\mathcal K_\Omega
		\qquad\text{in }\mathcal S'(\R^2),
		\]
		where \(\mathcal K_\Omega\) is the finite-part distribution associated with
		the homogeneous kernel \(K\). Consequently, for every
		\(f,g\in\mathcal S(\R)\),
		\[
		\sum_{j\in\mathbb Z}T_j(f,g)=T_\Omega(f,g)
		\qquad\text{in }\mathcal S'(\R),
		\]
		and the distribution on the right is the radial principal value in
		\eqref{eq:T-def}.
		
		More generally, suppose that, for some
		\(1<p_1,p_2,p<\infty\) with \(1/p=1/p_1+1/p_2\),
		\[
		\sum_{j\in\mathbb Z}
		\norm{T_j}_{L^{p_1}\times L^{p_2}\to L^p}<\infty.
		\]
	Then the series \(\sum_jT_j\) converges in operator norm, and its limit is
		the bounded extension of the radial principal-value operator \(T_\Omega\).
		In particular, the resulting operator is independent of the admissible
		cutoffs \(\chi\) and \(\eta\) used in the dyadic decomposition.
	\end{lemma}
	
	\begin{proof}
		By Lemma~\ref{lem:L1-from-K}, \(\Omega\in L^1(\Sph)\), so the finite-part
		distribution \(\mathcal K_\Omega\) defined in \eqref{eq:finite-part-L1} is
		available.  The radial partition of unity and the cancellation of \(\Omega\)
		give
		\[
		\sum_{|i|\le N}K^i\longrightarrow\mathcal K_\Omega
		\qquad\text{in }\mathcal S'(\R^2).
		\]
		Indeed, when testing against \(\Phi\in\mathcal S(\R^2)\), on \(0<r<1\)
		one may replace \(\Phi(r\theta)\) by
		\(\Phi(r\theta)-\Phi(0)\), while the integral over \(r>1\) is absolutely
		convergent.  Dominated convergence in the two radial integrals of
		\eqref{eq:finite-part-L1}, together with
		\(\sum_{i\in\mathbb Z}\chi(2^{-i}r)=1\), proves the assertion.
		
		Each annular kernel has integral zero:
		\[
		\int_{\R^2}K^i(y)\,dy
		=c_{\rm rad}
		\left(\int_0^\infty\chi(2^{-i}r)\frac{dr}{r}\right)
		\left(\int_{\Sph}\Omega\,d\sigma\right)=0.
		\]
		Hence \(\widehat{K^i}(0)=0\). Moreover,
		\[
		\widehat{K_j^i}(\xi)
		=\eta(2^{i-j}|\xi|)\widehat{K^i}(\xi),
		\]
		and therefore
		\[
		\sum_{j\in\mathbb Z}\widehat{K_j^i}(\xi)
		=\widehat{K^i}(\xi),
		\qquad \xi\in\R^2.
		\]
		For \(\xi\ne0\), this is the Littlewood--Paley partition of unity, while
		at \(\xi=0\) both sides vanish. The partial multiplier sums are uniformly
		bounded, so dominated convergence against Schwartz functions gives
		\[
		\sum_{j\in\mathbb Z}K_j^i=K^i
		\qquad\text{in }\mathcal S'(\R^2).
		\]
		
		It remains to justify the order of summation. By homogeneity,
		\[
		m_j(\xi):=\widehat{K_j}(\xi)
		=\sum_{i\in\mathbb Z}
		\eta(2^{i-j}|\xi|)\widehat{K^0}(2^i\xi).
		\]
		For fixed \(\xi\ne0\) and \(j\), only \(O(1)\) indices \(i\) occur in
		this sum. The low-frequency estimate \eqref{eq:K0-low}, the
		high-frequency estimate \eqref{eq:K0-high}, and the fixed-scale
		\(L^1\) bound give
		\[
		\norm{m_j}_{L^\infty(\R^2)}
		\lesssim
		\norm{\Omega}_{\cK_{1/2,\beta}}
		\begin{cases}
			2^j, & j\le0,\\
			1, & j=1,\\
			2^{-j/2}j^{-\beta}, & j\ge2.
		\end{cases}
		\]
		Thus \(\sum_jm_j\) converges absolutely in \(L^\infty\).  There is also
		pointwise absolute convergence before grouping: for fixed \(\xi\ne0\),
		\[
		\sum_{i\in\mathbb Z}|\widehat{K^0}(2^i\xi)|<\infty
		\]
		by the same low- and high-frequency estimates, and each \(i\) is paired
		with only \(O(1)\) indices \(j\) by the support of \(\eta\).  We may
		therefore rearrange the double series pointwise.  The uniform bound above
		then permits passage to \(\mathcal S'\), and the two reconstruction identities
		yield
		\[
		\sum_{j\in\mathbb Z}K_j
		=\sum_{i\in\mathbb Z}K^i
		=\mathcal K_\Omega
		\qquad\text{in }\mathcal S'(\R^2).
		\]
		
		For \(f,g,h\in\mathcal S(\R)\), put
		\[
		\Phi_{f,g,h}(y_1,y_2)
		:=\int_\R f(x-y_1)g(x-y_2)\overline{h(x)}\,dx.
		\]
		Then \(\Phi_{f,g,h}\in\mathcal S(\R^2)\), and the kernel reconstruction gives
		\[
		\lim_{N\to\infty}
		\left\langle\sum_{|j|\le N}T_j(f,g),h\right\rangle
		=\langle\mathcal K_\Omega,\Phi_{f,g,h}\rangle.
		\]
		By \eqref{eq:finite-part-pv-pairing}, the right-hand side is
		\[
		\langle T_\Omega(f,g),h\rangle.
		\]
		Hence
		\[
		\sum_{j\in\mathbb Z}T_j(f,g)=T_\Omega(f,g)
		\qquad\text{in }\mathcal S'(\R).
		\]
		
		Finally, if the operator norms are summable, the partial sums converge in
		\(L^{p_1}\times L^{p_2}\to L^p\) to an operator \(T\). Since \(L^p\)
		embeds continuously into \(\mathcal S'\), the distributional identity gives
		\(T(f,g)=T_\Omega(f,g)\) for Schwartz inputs. Density gives the asserted
		bounded extension. Any other admissible pair of cutoffs reconstructs the
		same intrinsic finite-part distribution, and therefore the same bounded
		operator.
	\end{proof}
	
	\subsection{Wavelet analysis of the high-frequency pieces}
	
	\subsubsection{Product wavelets and coefficient bounds}
	
	We use the compactly supported product-wavelet system of
	\cite{GrafakosHeHonzikPark2023}; see also
	\cite{Daubechies_1988,Meyer_1992}.  Fix a sufficiently large \(M\in\N\).
	The construction provides real-valued compactly supported functions
	\(\psi_F,\psi_M\in C^{M+2}(\R)\) and a constant \(C_0>1\) such that
	\begin{enumerate}
		\item \(\norm{\psi_F}_{L^2(\R)}=\norm{\psi_M}_{L^2(\R)}=1\);
		\item \(\int_\R x^\alpha\psi_M(x)\,dx=0\) for \(0\le \alpha\le M+1\);
		\item \(\supp\psi_F,\supp\psi_M\subset[-C_0,C_0]\).
	\end{enumerate}
	For \(\lambda\in\mathbb N_0\) set
	\[
	\cI_0:=\{F,M\}^2,
	\qquad
	\cI_\lambda:=\{F,M\}^2\setminus\{(F,F)\}\quad (\lambda\ge1).
	\]
	Only finitely many wavelet types occur, and all constants below are uniform in
	\(\vec G\in\cI_\lambda\).
	
	For \(G\in\{F,M\}\), \(\lambda\in\mathbb N_0\), and \(k\in\mathbb Z\), define the one-dimensional packets
	\[
	\psi_{G,k}^\lambda(\xi):=2^{\lambda/2}\psi_G(2^\lambda\xi-k).
	\]
	For \(\vec G=(G_1,G_2)\in\cI_\lambda\) and \(\vec k=(k_1,k_2)\in\mathbb Z^2\), define the product wavelets
	\[
	\Psi_{\vec G,\vec k}^\lambda(\xi_1,\xi_2)
	:=
	\psi_{G_1,k_1}^\lambda(\xi_1)\psi_{G_2,k_2}^\lambda(\xi_2).
	\]
	Then
	\[
	\bigl\{\Psi_{\vec G,\vec k}^\lambda:\lambda\in\mathbb N_0,\ \vec G\in\cI_\lambda,\ \vec k\in\mathbb Z^2\bigr\}
	\]
	forms an orthonormal basis of \(L^2(\R^2)\).
	
	We repeatedly use the support relation
	\[
	\supp\psi_{G,k}^\lambda\subset\bigl\{\xi\in\R: \abs{2^\lambda\xi-k}\le C_0\bigr\}.
	\]
	
	Let
	\[
	m_j^0(\vec\xi):=\widehat{K_j^0}(\vec\xi)\,\vartheta(2^{-j}\vec\xi),
	\]
	where \(\vartheta\in C_c^\infty(\R^2\setminus\{0\})\) is a fixed radial cutoff satisfying
	\[
	\supp \vartheta\subset\{\vec\xi\in\R^2:1/2\le |\vec\xi|\le 2\}
	\]
	and \(\vartheta \equiv 1\) on the support of \(\eta\).
	It follows that
	\begin{equation}\label{eq:annulus-j}
		\supp m_j^0
		\subset
		\{\vec\xi\in\R^2:2^{j-1}\le |\vec\xi|\le 2^{j+1}\}.
	\end{equation}
	
	Expand the frequency-localized multiplier \(m_j^0\) in this basis:
	\[
	m_j^0(\vec\xi)
	=
	\sum_{\lambda\ge0}\sum_{\vec G\in\cI_\lambda}\sum_{\vec k\in\mathbb Z^2}
	b_{\vec G,\vec k}^{\lambda,j}\Psi_{\vec G,\vec k}^\lambda(\vec\xi),
	\]
	where
	\[
	b_{\vec G,\vec k}^{\lambda,j}
	:=
	\int_{\R^2}m_j^0(\vec\xi)\Psi_{\vec G,\vec k}^\lambda(\vec\xi)\,d\vec\xi.
	\]
	
	By \eqref{eq:annulus-j} and the wavelet support condition,
	\(b_{\vec G,\vec k}^{\lambda,j}\) can be nonzero only if
	\begin{equation}\label{eq:k-annulus}
		2^{\lambda+j-1}-\sqrt2C_0
		\le |\vec k|\le
		2^{\lambda+j+1}+\sqrt2C_0.
	\end{equation}
	Thus \(|\vec k|\simeq2^{\lambda+j}\), with constants depending only on
	\(C_0\), whenever \(\lambda+j\) exceeds a fixed constant. The remaining
	finitely many values are treated separately below.
	
	\begin{lemma}\label{lem:coeff-bounds}
		Let \(j\ge2\). Then the wavelet coefficients satisfy:
		\begin{enumerate}
			\item for every \(\lambda\ge0\),
			\begin{equation}\label{eq:A-lambda-j}
				A_{\lambda,j}:=\sup_{\vec G\in\cI_\lambda}\norm{\{b_{\vec G,\vec k}^{\lambda,j}\}_{\vec k}}_{\ell^\infty}
				\lesssim
				\norm{\Omega}_{\cK_{1/2,\beta}}
				\,2^{-j/2}j^{-\beta}2^{-\lambda(M+3)};
			\end{equation}
			\item for every \(1<q<4\) and every \(\lambda\ge0\),
			\begin{equation}\label{eq:B-lambda-j}
				B_{\lambda,j,q}:=\sup_{\vec G\in\cI_\lambda}\norm{\{b_{\vec G,\vec k}^{\lambda,j}\}_{\vec k}}_{\ell^q}
				\lesssim
				\norm{\Omega}_{\cK_{1/2,\beta}}
				\,2^{-2\lambda(1/2-1/q)}2^{-j/2}2^{2j/q}j^{-\beta}.
			\end{equation}
		\end{enumerate}
	\end{lemma}
	
	\begin{proof}
		\emph{The \(\ell^\infty\) bound.}
		We use the standard coefficient estimate for compactly supported product
		wavelets with moments through order \(M+1\), applied to \(C^{M+2}\) functions in
		dimension two; see, for instance, \cite[Chapter~VI]{Meyer_1992} or
		\cite[Chapter~1]{Triebel_2006}. In the present normalization this yields
		\[
		\abs{b_{\vec G,\vec k}^{\lambda,j}}
		\lesssim
		2^{-\lambda(M+3)}
		\max_{\abs{\alpha}\le M+2}
		\norm{\p^\alpha m_j^0}_{L^\infty(\R^2)}.
		\]
		Here the Taylor remainder has order \(M+2\), and the two-dimensional
		\(L^2\) normalization contributes one additional power of \(2^{-\lambda}\),
		accounting for the exponent \(M+3\). Corollary~\ref{cor:Kj0-high} and the smooth
		cutoff \(\vartheta(2^{-j}\cdot)\) provide the required finite-order derivative
		bounds for \(m_j^0\). Hence
		\[
		\abs{b_{\vec G,\vec k}^{\lambda,j}}
		\lesssim
		\norm{\Omega}_{\cK_{1/2,\beta}}\,2^{-j/2}j^{-\beta}2^{-\lambda(M+3)}.
		\]
		Taking the supremum over \(\vec k\) and \(\vec G\) gives
		\eqref{eq:A-lambda-j}.
		
		\smallskip
		\noindent\emph{The \(\ell^q\) bound.}
		For compactly supported product wavelets, the wavelet square-function
		inequality, equivalently the wavelet characterization of \(F^0_{q,2}=L^q\),
		gives, uniformly in \(\lambda\) and the finitely many types
		\(\vec G\in\cI_\lambda\),
		\begin{equation}\label{eq:lq-wavelet-general}
			\left\|\Bigl(\sum_{\vec k\in\mathbb Z^2}
			\abs{b_{\vec G,\vec k}^{\lambda,j}\Psi_{\vec G,\vec k}^\lambda}^2\Bigr)^{1/2}\right\|_{L^q(\R^2)}
			\lesssim
			\norm{m_j^0}_{L^q(\R^2)};
		\end{equation}
		see \cite[Chapter~1]{Triebel_2006}. The compact support and uniformly bounded
		overlap of the wavelets also imply
		\[
		\norm{\{b_{\vec G,\vec k}^{\lambda,j}\}_{\vec k}}_{\ell^q}
		\lesssim
		2^{-\lambda(1-2/q)}
		\left\|\Bigl(\sum_{\vec k\in\mathbb Z^2}
		\abs{b_{\vec G,\vec k}^{\lambda,j}\Psi_{\vec G,\vec k}^\lambda}^2\Bigr)^{1/2}\right\|_{L^q(\R^2)}.
		\]
		Combining this with \eqref{eq:lq-wavelet-general} gives
		\[
		\norm{\{b_{\vec G,\vec k}^{\lambda,j}\}_{\vec k}}_{\ell^q}
		\lesssim
		2^{-2\lambda(1/2-1/q)}\norm{m_j^0}_{L^q(\R^2)}.
		\]
		It remains to estimate the \(L^q\) norm of \(m_j^0\). By \eqref{eq:annulus-j} and
		\eqref{eq:Kj0-pointwise},
		\begin{align*}
			\norm{m_j^0}_{L^q(\R^2)}
			&\le
			\Bigl(\int_{\{2^{j-1}\le\abs{\vec\xi}\le2^{j+1}\}}
			\abs{m_j^0(\vec\xi)}^q\,d\vec\xi\Bigr)^{1/q}
			\\
			&\lesssim
			\norm{\Omega}_{\cK_{1/2,\beta}}
			\Bigl(\int_{\{2^{j-1}\le\abs{\vec\xi}\le2^{j+1}\}}
			(2^{-j/2}j^{-\beta})^q\,d\vec\xi\Bigr)^{1/q}
			\\
			&\lesssim
			\norm{\Omega}_{\cK_{1/2,\beta}}
			\,2^{-j/2}j^{-\beta}2^{2j/q}.
		\end{align*}
		Substitution in the preceding inequality gives
		\eqref{eq:B-lambda-j}, with constants independent of \(j\) and \(\lambda\).
	\end{proof}
	
	\begin{remark}\label{rem:area-factor}
		The factor \(2^{2j/q}\) in \eqref{eq:B-lambda-j} is the \(q\)th root of
		the area of the two-dimensional annulus in \eqref{eq:annulus-j}. In the
		\(l=2\) regime below, raising this factor to the power \(q/4\) produces
		\(2^{j/2}\), which exactly cancels the \(2^{-j/2}\) decay in the endpoint
		Fourier estimate at the critical index \(a=\frac12\).
	\end{remark}
	
	\subsection[Model operator decomposition and the endpoint estimate]
	{Model operator decomposition and the endpoint estimate}
	
	Recall that
	\[
	m_j^0(\vec\xi)=\widehat{K_j^0}(\vec\xi)\vartheta(2^{-j}\vec\xi).
	\]
	The full multiplier piece is
	\[
	m_j(\vec\xi)
	=
	\sum_{i\in\mathbb Z} m_j^0(2^i\vec\xi)
	=
	\sum_{\gamma\in\mathbb Z} m_j^0(2^{-\gamma}\vec\xi),
	\]
	where the second identity follows by setting \(\gamma=-i\). Substitution of
	the wavelet expansion of \(m_j^0\) gives
	\[
	m_j(\vec\xi)
	=
	\sum_{\lambda\ge0}\sum_{\vec G\in\cI_\lambda}\sum_{\gamma\in\mathbb Z}
	\sum_{\vec k\in\cU^{\lambda+j}}
	b_{\vec G,\vec k}^{\lambda,j}
	\prod_{u=1}^{2}\psi_{G_u,k_u}^\lambda(2^{-\gamma}\xi_u),
	\]
	where, consistently with \eqref{eq:k-annulus}, we set
	\begin{equation}\label{eq:U-def}
		\cU^{\lambda+j}:=
		\left\{\vec k\in\mathbb Z^2:
		\max\{0,2^{\lambda+j-1}-\sqrt2C_0\}
		\le|\vec k|\le2^{\lambda+j+1}+\sqrt2C_0
		\right\}.
	\end{equation}
	This set contains every index corresponding to a nonzero coefficient.  After
	enlarging the fixed threshold \(J_0\) below if necessary, for \(j\ge J_0\) it
	is contained in the multiplicative annulus
	\(\{2^{\lambda+j-3}\le|\vec k|\le2^{\lambda+j+3}\}\) required by the model
	estimate.
	
	For \(G\in\{F,M\}\), \(k\in\mathbb Z\), \(\lambda\in\mathbb N_0\), and \(\gamma\in\mathbb Z\), define the localized one-dimensional operators
	\[
	L_{G,k}^{\lambda,\gamma}f
	:=
	\bigl(\psi_{G,k}^\lambda(2^{-\gamma}\cdot)\,\widehat f(\cdot)\bigr)^\vee.
	\]
	Then \(T_j\) admits the representation
	\begin{equation}\label{eq:model-decomp}
		T_j(f_1,f_2)
		=
		\sum_{\lambda\ge0}\sum_{\vec G\in\cI_\lambda}\sum_{\gamma\in\mathbb Z}
		\sum_{\vec k\in\cU^{\lambda+j}}
		b_{\vec G,\vec k}^{\lambda,j}
		\prod_{u=1}^{2}L_{G_u,k_u}^{\lambda,\gamma}f_u.
	\end{equation}
	
	Following Proposition~2.4 of
	\cite{GrafakosHeHonzikPark2023}, decompose \(\cU^{\lambda+j}\) according to the
	number of large coordinates.
	Fix \(N_0>2C_0\), and set
	\begin{align*}
		\cU_{1,1}^{\lambda+j}
		&:=\{\vec k\in\cU^{\lambda+j}:|k_1|\ge N_0>|k_2|\},\\
		\cU_{1,2}^{\lambda+j}
		&:=\{\vec k\in\cU^{\lambda+j}:|k_2|\ge N_0>|k_1|\},\\
		\cU_1^{\lambda+j}&:=
		\cU_{1,1}^{\lambda+j}\cup\cU_{1,2}^{\lambda+j},\\
		\cU_2^{\lambda+j}
		&:=\{\vec k\in\cU^{\lambda+j}:|k_1|\ge N_0,\ |k_2|\ge N_0\}.
	\end{align*}
	The two components of \(\cU_1^{\lambda+j}\) correspond to the \(l=1\) case
	of that proposition, with the coordinates interchanged for the second
	component; \(\cU_2^{\lambda+j}\) corresponds to the \(l=2\) case.
	
	If \(\max(|k_1|,|k_2|)<N_0\), the lower bound in
	\eqref{eq:U-def} forces \(\lambda+j\) to be bounded. Choose
	an integer \(J_0\ge2\), depending only on \(C_0\) and \(N_0\), so large that
	for \(j\ge J_0\) the region in which both coordinates are small is empty and,
	in the \(l=1\) region,
	the large packet frequency dominates the small one by a fixed factor. The
	finitely many pieces with \(2\le j<J_0\) satisfy the desired
	\(L^2\times L^2\to L^1\) bound by the fixed-scale bilinear
	Coifman--Meyer theorem, using the bounded derivatives of
	\(\widehat{K^0}\) and Lemma~\ref{lem:L1-from-K}. Their constants can be
	absorbed into one depending on \(\beta\).  We may therefore assume
	\(j\ge J_0\) in the model estimates below.
	
	The packet support relation is
	\begin{equation}\label{eq:packet-support-gamma}
		\supp\bigl(\psi_{G,k}^\lambda(2^{-\gamma}\cdot)\bigr)
		\subset
		\bigl\{\xi\in\R:
		|\xi-2^{\gamma-\lambda}k|\le C_0 2^{\gamma-\lambda}\bigr\}.
	\end{equation}
	The \(L^2\)-normalization of \(\psi_{G,k}^\lambda\) contributes a factor
	\(2^{\lambda/2}\). For a large coordinate, where \(|k|\ge N_0>C_0\), the
	fixed-\(k\) pointwise square sum of the packets is bounded by \(C\,2^\lambda\),
	rather than by an absolute constant.  We do not use this square sum, because
	the factor \(2^\lambda\) associated with the two-dimensional product-wavelet
	normalization is already incorporated in Proposition~2.4 of
	\cite{GrafakosHeHonzikPark2023}.  Instead, that proposition requires a
	square-summable family of frequency truncations of the inputs; we now construct
	this family.
	
	There exist constants \(0<c_1<C_1<\infty\), depending only on
	\(C_0,N_0\), such that every packet in a coordinate declared large is
	supported in
	\[
	\{\xi:c_1 2^{\gamma-\lambda}\le |\xi|
	\le C_1 2^{\gamma+j}\}.
	\]
	Choose smooth functions \(\rho_{\lambda,\gamma,j}\), uniformly bounded by
	one, equal to one on this annulus and supported in the fixed enlargement
	\begin{equation}\label{eq:rho-support}
		\supp\rho_{\lambda,\gamma,j}
		\subset
		\{\xi:c_2 2^{\gamma-\lambda}\le |\xi|
		\le C_2 2^{\gamma+j}\},
	\end{equation}
	where \(0<c_2<c_1<C_1<C_2\). For \(u=1,2\), define
	\begin{equation}\label{eq:f-lambda-gamma-j}
		\widehat{f_u^{\lambda,\gamma,j}}(\xi)
		:=\rho_{\lambda,\gamma,j}(\xi)\widehat f_u(\xi).
	\end{equation}
	Whenever the \(u\)-th coordinate is large, we then have
	\begin{equation}\label{eq:L-input-truncation}
		L_{G_u,k_u}^{\lambda,\gamma}f_u
		=L_{G_u,k_u}^{\lambda,\gamma}f_u^{\lambda,\gamma,j}.
	\end{equation}
	For each fixed \(\xi\ne0\), condition \eqref{eq:rho-support} holds for at
	most \(C(\lambda+j+1)\) integers \(\gamma\). Plancherel's theorem therefore
	gives
	\begin{align}
		\sum_{\gamma\in\mathbb Z}
		\norm{f_u^{\lambda,\gamma,j}}_{L^2(\R)}^2
		&\lesssim(\lambda+j+1)\norm{f_u}_{L^2(\R)}^2\notag\\
		&\lesssim j(\lambda+1)\norm{f_u}_{L^2(\R)}^2,
		\qquad j\ge2,
		\label{eq:f-square-gamma}
	\end{align}
	uniformly in \(\lambda\).  This is the square-function loss used below.
	
	For clarity, we isolate the precise form of
	\cite[Proposition~2.4]{GrafakosHeHonzikPark2023} used in this proof.
	Fix \(j\ge J_0\), \(\lambda\ge0\), and
	\(\vec G=(G_1,G_2)\in\cI_\lambda\).  For \(s=1,2\), let
	\(s^\ast=3-s\), and define
	\begin{align*}
		T_{j,\lambda,\gamma,\vec G}^{(1,s)}(f_1,f_2)
		&:=
		\sum_{\vec k\in\cU_{1,s}^{\lambda+j}}
		b_{\vec G,\vec k}^{\lambda,j}
		\prod_{u=1}^2L_{G_u,k_u}^{\lambda,\gamma}f_u,\\
		T_{j,\lambda,\gamma,\vec G}^{(1)}
		&:=
		T_{j,\lambda,\gamma,\vec G}^{(1,1)}
		+T_{j,\lambda,\gamma,\vec G}^{(1,2)},\\
		T_{j,\lambda,\gamma,\vec G}^{(2)}(f_1,f_2)
		&:=
		\sum_{\vec k\in\cU_2^{\lambda+j}}
		b_{\vec G,\vec k}^{\lambda,j}
		\prod_{u=1}^2L_{G_u,k_u}^{\lambda,\gamma}f_u.
	\end{align*}
	Thus the superscript \(l\) records the number of large coordinates of
	\(\vec k\). Exactly one input is frequency truncated in the \(l=1\)
	sector, whereas both inputs are truncated in the \(l=2\) sector.
	
	\begin{lemma}[Bilinear model estimates]
		\label{lem:GHHP-specialized}
		With the notation above, the following estimates hold uniformly in
		\(j\ge J_0\), \(\lambda\ge0\), and
		\(\vec G\in\cI_\lambda\).
		\begin{enumerate}
			\item For \(s=1,2\),
			\begin{align}
				&\left\|
				\left(
				\sum_{\gamma\in\mathbb Z}
				\left|T_{j,\lambda,\gamma,\vec G}^{(1,s)}(f_1,f_2)\right|^2
				\right)^{1/2}
				\right\|_{L^1(\R)}
				\notag\\
				&\qquad\lesssim
				A_{\lambda,j}\,2^\lambda
				\left(
				\sum_{\gamma\in\mathbb Z}
				\norm{f_s^{\lambda,\gamma,j}}_{L^2(\R)}^2
				\right)^{1/2}
				\norm{f_{s^\ast}}_{L^2(\R)}.
				\label{eq:GHHP-one-large}
			\end{align}
			Consequently,
			\begin{equation}
				\left\|
				\left(
				\sum_{\gamma\in\mathbb Z}
				\left|T_{j,\lambda,\gamma,\vec G}^{(1)}(f_1,f_2)\right|^2
				\right)^{1/2}
				\right\|_{L^1}
				\lesssim
				A_{\lambda,j}2^\lambda
				j^{1/2}(\lambda+1)^{1/2}
				\prod_{u=1}^2\norm{f_u}_{L^2}.
				\label{eq:unbal-vector}
			\end{equation}
			
			\item For every \(1<q<4\),
			\begin{align}
				&\left\|
				\sum_{\gamma\in\mathbb Z}
				T_{j,\lambda,\gamma,\vec G}^{(2)}(f_1,f_2)
				\right\|_{L^1(\R)}
				\notag\\
				&\qquad\lesssim
				A_{\lambda,j}^{1-q/4}B_{\lambda,j,q}^{q/4}
				\,2^\lambda
				\prod_{u=1}^2
				\left(
				\sum_{\gamma\in\mathbb Z}
				\norm{f_u^{\lambda,\gamma,j}}_{L^2(\R)}^2
				\right)^{1/2}
				\notag\\
				&\qquad\lesssim
				A_{\lambda,j}^{1-q/4}B_{\lambda,j,q}^{q/4}
				\,2^\lambda j(\lambda+1)
				\prod_{u=1}^2\norm{f_u}_{L^2(\R)}.
				\label{eq:GHHP-two-large}
			\end{align}
		\end{enumerate}
		The sums over \(\gamma\) may initially be taken over finite sets; the
		displayed estimates then provide the convergence required below for
		Schwartz inputs.
	\end{lemma}
	
	\begin{proof}[Verification of the specialization]
		We check only the correspondence between the notation of
		\cite[Proposition~2.4]{GrafakosHeHonzikPark2023} and the present
		decomposition; no part of the cited proposition is reproved here.
		
		For fixed \(j,\lambda,\vec G\), use in that proposition the coefficient
		family
		\[
		b_{\vec k}^{\lambda,\gamma,\mu}
		:=b_{\vec G,\vec k}^{\lambda,j},
		\qquad \mu=j.
		\]
		Although this family is independent of \(\gamma\), the cited result allows
		general \(\gamma\)-dependent coefficients. Lemma~\ref{lem:coeff-bounds}
		therefore supplies the required uniform bounds
		\[
		\sup_{\gamma\in\mathbb Z}
		\norm{\{b_{\vec k}^{\lambda,\gamma,j}\}_{\vec k}}_{\ell^\infty}
		\le A_{\lambda,j},
		\qquad
		\sup_{\gamma\in\mathbb Z}
		\norm{\{b_{\vec k}^{\lambda,\gamma,j}\}_{\vec k}}_{\ell^q}
		\le B_{\lambda,j,q}.
		\]
		By the choice of \(J_0\), the set \(\cU^{\lambda+j}\) lies in the
		multiplicative lattice annulus required there. Its subsets
		\(\cU_{1,1}^{\lambda+j}\), \(\cU_{1,2}^{\lambda+j}\), and
		\(\cU_2^{\lambda+j}\) are precisely the sectors with one or two large
		coordinates.
		
		For \(\cU_{1,1}^{\lambda+j}\), apply part~(1) of the cited proposition with
		\[
		m=2,\qquad n=1,\qquad \mu=j,\qquad l=1,\qquad r=2.
		\]
		The first coordinate is large, so \eqref{eq:L-input-truncation} permits the
		replacement of \(f_1\) by \(f_1^{\lambda,\gamma,j}\); the second input is
		left unchanged. Since
		\[
		2^{\lambda mn/2}=2^\lambda,
		\]
		the resulting estimate is \eqref{eq:GHHP-one-large} with \(s=1\).
		Interchanging the coordinates treats \(\cU_{1,2}^{\lambda+j}\). The
		vector-valued triangle inequality combines the two components, and
		\eqref{eq:f-square-gamma} then gives \eqref{eq:unbal-vector}.
		
		For \(\cU_2^{\lambda+j}\), apply part~(2) of the cited proposition with
		\[
		m=l=2,\qquad n=1,\qquad \mu=j.
		\]
		Its admissible coefficient exponent satisfies
		\[
		0<q<\frac{2l}{l-1}=4.
		\]
		The estimate \eqref{eq:B-lambda-j} is available for \(q>1\), so we choose
		\(1<q<4\). Both coordinates are large and
		\eqref{eq:L-input-truncation} is used in both inputs. The interpolation
		exponent on the coefficient bounds becomes
		\[
		\frac{(l-1)q}{2l}=\frac q4,
		\]
		while the wavelet normalization again contributes
		\(2^{\lambda mn/2}=2^\lambda\). This gives the first inequality in
		\eqref{eq:GHHP-two-large}; applying \eqref{eq:f-square-gamma} to the two
		inputs gives the second one.
	\end{proof}
	
	\begin{remark}[The roles of the two frequency sectors]
		\label{rem:GHHP-bookkeeping}
		The conclusions of Lemma~\ref{lem:GHHP-specialized} have different forms
		for a structural reason. In the \(l=1\) sector, the cited proposition uses
		only the \(\ell^\infty\)-coefficient bound \(A_{\lambda,j}\) and produces
		a square function in the dilation parameter \(\gamma\). Only the input
		corresponding to the large coordinate is truncated, so the overlap loss is
		\(j^{1/2}(\lambda+1)^{1/2}\). In the \(l=2\) sector, the proposition
		interpolates the \(\ell^\infty\)- and \(\ell^q\)-coefficient bounds and
		produces
		\[
		A_{\lambda,j}^{1-q/4}B_{\lambda,j,q}^{q/4}.
		\]
		It controls the ordinary sum over \(\gamma\) directly. Both inputs are
		truncated, and the two square-root overlap losses multiply to
		\(j(\lambda+1)\). In both sectors, \(2^\lambda\) is the specialization of
		\(2^{\lambda mn/2}\) to \(m=2,n=1\); it is not an additional
		frequency-overlap loss.
	\end{remark}
	
	It remains to convert the vector-valued estimate in the \(l=1\) sector
	into an estimate for the ordinary \(\gamma\)-sum. Put
	\[
	h_{\gamma,\vec G}
	:=T_{j,\lambda,\gamma,\vec G}^{(1)}(f_1,f_2).
	\]
	In \(\cU_1^{\lambda+j}\), the large index has size comparable to
	\(2^{\lambda+j}\), whereas the other index is bounded by \(N_0\).
	The packet support relation \eqref{eq:packet-support-gamma} and the choice
	of \(J_0\) imply
	\begin{equation}
		\supp\widehat h_{\gamma,\vec G}
		\subset
		\{\zeta\in\R:
		c_3 2^{\gamma+j}\le|\zeta|\le C_3 2^{\gamma+j}\},
		\label{eq:unbal-output-annulus}
	\end{equation}
	with constants independent of \(j,\lambda,\gamma\), and \(\vec G\).
	The Fourier-support synthesis inequality for
	\(H^1(\R)=F^0_{1,2}(\R)\) therefore gives, first for finite sets of
	\(\gamma\),
	\[
	\left\|\sum_\gamma h_{\gamma,\vec G}\right\|_{L^1}
	\lesssim
	\left\|\sum_\gamma h_{\gamma,\vec G}\right\|_{H^1}
	\lesssim
	\left\|
	\left(\sum_\gamma|h_{\gamma,\vec G}|^2\right)^{1/2}
	\right\|_{L^1};
	\]
	see \cite[Chapter~1]{Triebel_2006}.
	
	The same estimate applied to a tail set \(E\subset\mathbb Z\), together
	with \eqref{eq:GHHP-one-large}, has on its right-hand side the corresponding
	tail of
	\(\sum_\gamma\norm{f_s^{\lambda,\gamma,j}}_2^2\). This tail tends to zero
	because the cutoffs in \eqref{eq:rho-support} have finite overlap and the
	inputs are in \(L^2\). Hence the finite sums converge in \(H^1\), and
	therefore in \(L^1\). The possible polynomial ambiguity in homogeneous
	Littlewood--Paley synthesis vanishes because the reconstructed distribution
	belongs to \(L^1(\R)\), whereas no nonzero polynomial does.
	
	Define the full sector contributions, including the finite sum over wavelet
	types, by
	\[
	T_{j,\lambda}^{(l)}
	:=
	\sum_{\vec G\in\cI_\lambda}
	\sum_{\gamma\in\mathbb Z}
	T_{j,\lambda,\gamma,\vec G}^{(l)},
	\qquad l=1,2.
	\]
	Combining the synthesis step with \eqref{eq:unbal-vector} gives
	\begin{equation}
		\norm{T_{j,\lambda}^{(1)}(f_1,f_2)}_{L^1(\R)}
		\lesssim
		A_{\lambda,j}\,2^\lambda
		j^{1/2}(\lambda+1)^{1/2}
		\norm{f_1}_{L^2(\R)}\norm{f_2}_{L^2(\R)}.
		\label{eq:unbal-model}
	\end{equation}
	For \(l=2\), no additional synthesis is needed; summing
	\eqref{eq:GHHP-two-large} over the uniformly finite set
	\(\cI_\lambda\) yields
	\begin{equation}
		\norm{T_{j,\lambda}^{(2)}(f_1,f_2)}_{L^1(\R)}
		\lesssim
		A_{\lambda,j}^{1-q/4}B_{\lambda,j,q}^{q/4}
		\,2^\lambda j(\lambda+1)
		\norm{f_1}_{L^2(\R)}\norm{f_2}_{L^2(\R)},
		\label{eq:bal-model}
	\end{equation}
	for every fixed \(1<q<4\).
	
	\begin{proposition}\label{prop:L2-endpoint}
		For every \(j\ge2\) one has
		\[
		\norm{T_j(f_1,f_2)}_{L^1(\R)}
		\lesssim
		j^{1-\beta}\norm{\Omega}_{\cK_{1/2,\beta}}
		\norm{f_1}_{L^2(\R)}\norm{f_2}_{L^2(\R)}.
		\]
	\end{proposition}
	
	\begin{proof}
		The model estimates \eqref{eq:unbal-model} and \eqref{eq:bal-model} were
		proved for \(j\ge J_0\). For the finitely many indices \(2\le j<J_0\), the
		fixed-scale Coifman--Meyer argument above gives the asserted estimate after
		enlarging the implicit constant, which may depend on \(\beta\). Indeed,
		\(j^{1-\beta}\) is bounded below by a positive \(\beta\)-dependent constant
		on this finite set. We may therefore assume \(j\ge J_0\).
		
		By \eqref{eq:model-decomp}, it suffices to sum
		\(T_{j,\lambda}^{(1)}+T_{j,\lambda}^{(2)}\) over \(\lambda\ge0\).
		
		\smallskip
		\noindent\emph{The \(l=1\) sector.}
		Combining \eqref{eq:unbal-model} and \eqref{eq:A-lambda-j}, we obtain
		\[
		\norm{T_{j,\lambda}^{(1)}}_{L^2\times L^2\to L^1}
		\lesssim
		A_{\lambda,j}2^\lambda j^{1/2}(\lambda+1)^{1/2}
		\lesssim
		2^{-j/2}j^{1/2-\beta}2^{-\lambda(M+2)}(\lambda+1)^{1/2}
		\norm{\Omega}_{\cK_{1/2,\beta}}.
		\]
		The right-hand side is summable in \(\lambda\), and its sum is bounded by
		\(j^{1-\beta}\norm{\Omega}_{\cK_{1/2,\beta}}\).
		
		\smallskip
		\noindent\emph{The \(l=2\) sector.}
		Fix \(1<q<4\). By \eqref{eq:bal-model}, \eqref{eq:A-lambda-j}, and
		\eqref{eq:B-lambda-j},
		\begin{align*}
			\norm{T_{j,\lambda}^{(2)}}_{L^2\times L^2\to L^1}
			&\lesssim
			A_{\lambda,j}^{1-q/4}B_{\lambda,j,q}^{q/4}2^\lambda j(\lambda+1)
			\\
			&\lesssim
			\Bigl(2^{-j/2}j^{-\beta}2^{-\lambda(M+3)}\Bigr)^{1-q/4}
			\\
			&\quad\times
			\Bigl(2^{-2\lambda(1/2-1/q)}
			2^{-j/2}2^{2j/q}j^{-\beta}\Bigr)^{q/4}
			2^\lambda j(\lambda+1)
			\norm{\Omega}_{\cK_{1/2,\beta}}.
		\end{align*}
		Collecting the \(j\)- and \(\lambda\)-dependent factors gives
		\begin{align*}
			\norm{T_{j,\lambda}^{(2)}}_{L^2\times L^2\to L^1}
			&\lesssim
			\Bigl(2^{-j/2}\Bigr)^{1-q/4}
			\Bigl(2^{-j/2}\Bigr)^{q/4}
			\Bigl(2^{2j/q}\Bigr)^{q/4}
			\,j^{1-\beta}
			\,(\lambda+1)\,
			2^{-\lambda C_{M,q}}
			\norm{\Omega}_{\cK_{1/2,\beta}},
		\end{align*}
		where
		\[
		C_{M,q}:=(M+3)\Bigl(1-\frac q4\Bigr)+2\Bigl(\frac12-\frac1q\Bigr)\frac q4-1.
		\]
		Moreover,
		\[
		\Bigl(2^{-j/2}\Bigr)^{1-q/4}
		\Bigl(2^{-j/2}\Bigr)^{q/4}
		\Bigl(2^{2j/q}\Bigr)^{q/4}
		=2^{-j/2}2^{j/2}=1.
		\]
		Therefore
		\[
		\norm{T_{j,\lambda}^{(2)}}_{L^2\times L^2\to L^1}
		\lesssim
		j^{1-\beta}(\lambda+1)2^{-\lambda C_{M,q}}
		\norm{\Omega}_{\cK_{1/2,\beta}}.
		\]
		Since \(q<4\), we may choose \(M\) sufficiently large that \(C_{M,q}>0\).
		The \(\lambda\)-sum therefore converges.
		
		Summing the \(l=1\) and \(l=2\) contributions over \(\lambda\ge0\) gives
		\[
		\norm{T_j}_{L^2\times L^2\to L^1}
		\lesssim
		j^{1-\beta}\norm{\Omega}_{\cK_{1/2,\beta}},
		\]
		as claimed.
	\end{proof}
	
	\subsection{Interpolation and summation}
	
	We combine the high-frequency endpoint estimate of
	Proposition~\ref{prop:L2-endpoint},
	\[
	\norm{T_j}_{L^2(\R)\times L^2(\R)\to L^1(\R)}
	\lesssim
	j^{1-\beta}\norm{\Omega}_{\cK_{1/2,\beta}}
	\qquad (j\ge2),
	\]
	with the following polynomial estimate for positive-scale rough bilinear
	pieces.
	
	\begin{lemma}\label{lem:poly-growth}
		Let \(1<r_1,r_2,r<\infty\) with
		\(1/r=1/r_1+1/r_2\), and let
		\(\Omega\in L^1(\Sph)\) have mean zero. Then, for every \(j\ge1\),
		\[
		\norm{T_j(f_1,f_2)}_{L^r(\R)}
		\lesssim
		j^2\norm{\Omega}_{L^1(\Sph)}
		\norm{f_1}_{L^{r_1}(\R)}
		\norm{f_2}_{L^{r_2}(\R)}.
		\]
		If, in addition, \(\Omega\in\cK_{1/2,\beta}(\Sph)\), then
		\[
		\norm{T_j}_{L^{r_1}\times L^{r_2}\to L^r}
		\lesssim_\beta
		j^2\norm{\Omega}_{\cK_{1/2,\beta}}.
		\]
	\end{lemma}
	
	\begin{proof}
		Proposition~4 of \cite{dosidis_multilinear_2024} gives an estimate with loss
		\(C_\varepsilon2^{\varepsilon j}\).  Adapting its shifted maximal- and
		square-function argument to \(m=2\), \(n=1\), and the cutoffs used here yields
		the polynomial estimate needed below.  The two input-coordinate shifts
		produce two logarithmic factors, which lead to the loss \(j^2\).  We include
		the details to verify both the cutoff compatibility and this polynomial
		dependence.
		
		Let \(r_3=r'\), and for Schwartz functions set
		\[
		\Lambda_j(f_1,f_2,f_3)
		:=\int_\R T_j(f_1,f_2)(x)f_3(x)\,dx.
		\]
		By homogeneity,
		\[
		\widehat{K_j}(\xi_1,\xi_2)
		=\sum_{k\in\mathbb Z}
		\widehat{K_j^0}(2^k\xi_1,2^k\xi_2).
		\]
		The summand indexed by \(k\) is supported where
		\(|(\xi_1,\xi_2)|\simeq2^{j-k}\). Writing
		\(\xi_3=-\xi_1-\xi_2\), at least two among
		\(\xi_1,\xi_2,\xi_3\) then have magnitude comparable to
		\(2^{j-k}\). A finite smooth partition of unity reduces
		\(\Lambda_j\) to forms of the type
		\[
		\begin{aligned}
			\Lambda_{j,\boldsymbol\phi}
			=\sum_{k\in\mathbb Z}\int_{\R^2}
			&\widehat{K_j^0}(2^k\xi_1,2^k\xi_2)
			\prod_{\ell=1}^3
			\bigl[
			\widehat f_\ell(\xi_\ell)
			\widehat\phi_\ell(2^{k-j}\xi_\ell)
			\bigr] \,d\xi_1\,d\xi_2.
		\end{aligned}
		\]
		where the \(\widehat\phi_\ell\) are compactly supported and at least two
		of them are supported away from the origin. Only finitely many such forms
		are needed.
		
		For \(t>0\) and \(v\in\R\), write
		\[
		(\phi_\ell)_t^v(x):=t\phi_\ell(tx-v),
		\]
		and set \(y_3=0\). Fourier inversion and the change of variables used in
		the proof of Proposition~4 in \cite{dosidis_multilinear_2024} give
		\[
		\begin{aligned}
			\Lambda_{j,\boldsymbol\phi}
			=\int_{\R^2}K_j^0(y)
			\sum_{k\in\mathbb Z}\int_\R
			\prod_{\ell=1}^3
			\bigl[
			f_\ell*(\phi_\ell)_{2^{j-k}}^{2^jy_\ell}
			\bigr](x)\,dx\,dy.
		\end{aligned}
		\]
		Let \(a\ne b\) be two indices for which
		\(\widehat\phi_a\) and \(\widehat\phi_b\) are supported away from zero.
		For fixed \(y\), Cauchy--Schwarz in \(k\), followed by H\"older's
		inequality in \(x\), bounds the inner sum by
		\[
		\begin{aligned}
			&\prod_{\ell\notin\{a,b\}}
			\norm{\sup_k
				|f_\ell*(\phi_\ell)_{2^{j-k}}^{2^jy_\ell}|}_{L^{r_\ell}}
			\\
			&\qquad\times
			\prod_{\ell\in\{a,b\}}
			\norm{\left(\sum_k
				|f_\ell*(\phi_\ell)_{2^{j-k}}^{2^jy_\ell}|^2
				\right)^{1/2}}_{L^{r_\ell}}.
		\end{aligned}
		\]
		Here \(r_3=r'\). The shifted maximal and square-function estimates, in
		the form used in \cite[the proof of Proposition~4]{dosidis_multilinear_2024},
		show that the last display is at most
		\[
		C_{r_1,r_2}
		\prod_{i=1}^2\log(2+2^j|y_i|)
		\prod_{\ell=1}^3\norm{f_\ell}_{L^{r_\ell}}.
		\]
		The third shift is zero and contributes only an absolute constant. For
		\(j\ge1\),
		\[
		\log(2+2^j|y_i|)
		\lesssim j\log(2+|y_i|),
		\qquad i=1,2.
		\]
		Consequently,
		\begin{equation}\label{eq:poly-growth-reduction}
			|\Lambda_{j,\boldsymbol\phi}(f_1,f_2,f_3)|
			\lesssim
			j^2 W_j
			\prod_{\ell=1}^3\norm{f_\ell}_{L^{r_\ell}},
		\end{equation}
		where
		\[
		W_j:=\int_{\R^2}|K_j^0(y)|
		\prod_{i=1}^2\log(2+|y_i|)\,dy.
		\]
		
		We next verify that
		\(W_j\lesssim\norm{\Omega}_{L^1}\), uniformly for \(j\ge1\). Let
		\[
		\psi:=\bigl(\eta(|\cdot|)\bigr)^\vee,
		\qquad
		\psi_j(y):=2^{2j}\psi(2^jy).
		\]
		Then \(K_j^0=K^0*\psi_j\). Since \(\psi\) is a Schwartz function,
		\[
		|\psi(x)|
		\lesssim
		\sum_{\nu=0}^\infty
		2^{-4\nu}\one_{B(0,2^\nu)}(x),
		\]
		and therefore
		\[
		|\psi_j(x)|
		\lesssim
		\sum_{\nu=0}^\infty
		2^{2j-4\nu}\one_{B(0,2^{\nu-j})}(x).
		\]
		The kernel \(K^0\) is supported in a fixed annulus. Hence
		\(|K^0|*\one_{B(0,2^{\nu-j})}\) is supported in
		\(B(0,C2^\nu)\) for \(j\ge1\), and on this ball
		\[
		\prod_{i=1}^2\log(2+|y_i|)\lesssim(\nu+1)^2.
		\]
		It follows that
		\begin{align*}
			W_j
			&\lesssim
			\sum_{\nu=0}^\infty
			2^{2j-4\nu}(\nu+1)^2
			\norm{|K^0|*\one_{B(0,2^{\nu-j})}}_{L^1(\R^2)}
			\\
			&\lesssim
			\norm{K^0}_{L^1(\R^2)}
			\sum_{\nu=0}^\infty(\nu+1)^2 2^{-2\nu}
			\lesssim
			\norm{\Omega}_{L^1(\Sph)}.
		\end{align*}
		In the last step we used polar coordinates and the fixed annular support of
		\(\chi\), which give
		\[
		\norm{K^0}_{L^1(\R^2)}
		\lesssim\norm{\Omega}_{L^1(\Sph)}.
		\]
		Substituting this estimate into \eqref{eq:poly-growth-reduction}, summing
		the finitely many frequency-partition pieces, and applying duality proves
		\[
		\norm{T_j}_{L^{r_1}\times L^{r_2}\to L^r}
		\lesssim
		j^2\norm{\Omega}_{L^1(\Sph)}
		\]
		for all sufficiently large \(j\).  The finitely many remaining positive
		scales follow from the fixed-scale bilinear Coifman--Meyer theorem, after
		enlarging the constant. Finally, under the additional directional
		hypothesis, Lemma~\ref{lem:L1-from-K} gives
		\[
		\norm{\Omega}_{L^1(\Sph)}
		\lesssim_\beta
		\norm{\Omega}_{\cK_{1/2,\beta}},
		\]
		which proves the second assertion.
	\end{proof}
	
	We use the following strong-type consequence of Sagher's multilinear analytic
	interpolation theorem \cite{Sagher1969}; this restricted form suffices here.
	
	\begin{lemma}\label{lem:sagher}
		Let \(S\) be a bilinear operator. Assume that, for \(\nu=0,1\),
		\[
		\norm{S(f_1,f_2)}_{L^{s_\nu}}
		\le A_\nu
		\norm{f_1}_{L^{r_{1,\nu}}}
		\norm{f_2}_{L^{r_{2,\nu}}},
		\]
		where all displayed spaces are Banach spaces.  For
		\(0<\vartheta<1\), define
		\[
		\frac1{r_u}
		=\frac{1-\vartheta}{r_{u,0}}+\frac{\vartheta}{r_{u,1}}
		\quad(u=1,2),
		\qquad
		\frac1s=\frac{1-\vartheta}{s_0}+\frac{\vartheta}{s_1}.
		\]
		Then
		\[
		\norm{S(f_1,f_2)}_{L^s}
		\le C A_0^{1-\vartheta}A_1^\vartheta
		\norm{f_1}_{L^{r_1}}\norm{f_2}_{L^{r_2}}.
		\]
	\end{lemma}
	
	\begin{proof}[Proof of Theorem~\ref{thm:Klog-main}]
		\emph{Step 1: low and middle scales.}
		By Proposition~\ref{prop:negative-scales}, the series \(\sum_{j\le0}T_j\)
		converges in operator norm on
		\(L^{p_1}(\R)\times L^{p_2}(\R)\to L^p(\R)\). Lemma~\ref{lem:poly-growth}
		controls the single middle scale \(j=1\), which may be absorbed into the final
		constant. It remains to sum \(T_j\) over \(j\ge2\).
		
		\smallskip
		\noindent\emph{Step 2: choice of auxiliary exponents.}
		Define the critical interpolation parameter
		\begin{equation}\label{eq:theta-star}
			\theta_* := \min\left\{ \frac{2}{p_1}, \frac{2}{p_1'}, \frac{2}{p_2}, \frac{2}{p_2'} \right\} = \frac{2}{\max\{p_1, p_1', p_2, p_2'\}}.
		\end{equation}
		Fix \(\theta\in(0,\theta_*)\), and define \(r_1,r_2,r\) by
		\begin{equation}\label{eq:aux-exponents}
			\frac1{p_i}=(1-\theta)\frac1{r_i}+\theta\frac12
			\quad (i=1,2),
			\qquad
			\frac1p=(1-\theta)\frac1r+\theta.
		\end{equation}
		Equivalently,
		\begin{equation}\label{eq:aux-explicit}
			r_i=\frac{2(1-\theta)p_i}{2-\theta p_i},
			\qquad
			r=\frac{(1-\theta)p}{1-\theta p}.
		\end{equation}
		We verify that \(1<r_1,r_2,r<\infty\). First,
		\[
		r_i<\infty \iff \theta<\frac2{p_i},
		\qquad
		r_i>1 \iff \theta<\frac2{p_i'}.
		\]
		Both conditions follow because \(\theta<\theta_*\) and \(\theta_*\) is defined
		by \eqref{eq:theta-star}.
		Next,
		\[
		r<\infty
		\iff
		\theta<\frac1p.
		\]
		Moreover,
		\[
		\theta<\theta_*
		\le
		\min\Bigl\{\frac2{p_1},\frac2{p_2}\Bigr\}
		\le
		\frac1{p_1}+\frac1{p_2}
		=
		\frac1p,
		\]
		so \(r<\infty\). Finally,
		\[
		\frac1r=\frac{1/p-\theta}{1-\theta}<1
		\]
		because \(p>1\), and hence \(r>1\).
		
		\smallskip
		\noindent\emph{Step 3: interpolation for each fixed \(j\ge2\).}
		Apply Lemma~\ref{lem:sagher} to \(S=T_j\) with interpolation parameter
		\(\vartheta=\theta\) and the following endpoints:
		\begin{itemize}
			\item endpoint \(0\): the polynomial growth estimate of
			Lemma~\ref{lem:poly-growth} on
			\[
			L^{r_1}(\R)\times L^{r_2}(\R)\to L^r(\R);
			\]
			\item endpoint \(1\): the \(L^2\times L^2\to L^1\) estimate of
			Proposition~\ref{prop:L2-endpoint}, viewed as
			\[
			L^2(\R)\times L^2(\R)\to L^1(\R).
			\]
		\end{itemize}
		The interpolation identities are precisely \eqref{eq:aux-exponents}. Hence
		\begin{equation}\label{eq:interp-j}
			\norm{T_j(f_1,f_2)}_{L^p(\R)}
			\lesssim
			(j^2)^{1-\theta}(j^{1-\beta})^{\theta}
			\norm{\Omega}_{\cK_{1/2,\beta}}
			\norm{f_1}_{L^{p_1}(\R)}\norm{f_2}_{L^{p_2}(\R)}.
		\end{equation}
		Since
		\[
		(j^2)^{1-\theta}(j^{1-\beta})^\theta
		=
		j^{2(1-\theta)+\theta(1-\beta)}
		=
		j^{2-\theta-\beta\theta},
		\]
		we obtain
		\begin{equation}\label{eq:interp-j-simplified}
			\norm{T_j}_{L^{p_1}\times L^{p_2}\to L^p}
			\lesssim
			j^{2-\theta-\beta\theta}
			\norm{\Omega}_{\cK_{1/2,\beta}}.
		\end{equation}
		
		\smallskip
		\noindent\emph{Step 4: summation over \(j\ge2\).}
		The series \(\sum_{j\ge2}T_j\) converges in operator norm provided that
		\[
		\sum_{j\ge2} j^{2-\theta-\beta\theta}<\infty.
		\]
		This condition is equivalent to
		\begin{equation}\label{eq:beta-condition-final}
			2-\theta-\beta\theta<-1.
		\end{equation}
		Equivalently,
		\begin{equation}\label{eq:beta-condition-theta}
			\beta>\frac{3-\theta}{\theta}=\frac3\theta-1.
		\end{equation}
		Thus, for each fixed admissible interpolation parameter
		\(\theta\in(0,\theta_*)\), the present argument requires
		\[
		\beta>\frac3\theta-1.
		\]
		We do not claim that this threshold is sharp.
		
		Finally, the function \(\theta\mapsto \frac3\theta-1\) is strictly decreasing
		on \((0,1)\).  The least restrictive condition furnished by this argument is
		therefore the limiting value
		\[
		\lim_{\theta\uparrow\theta_*}\Bigl(\frac3\theta-1\Bigr)
		=
		\frac3{\theta_*}-1.
		\]
		Consequently, if
		\begin{equation}\label{eq:beta-threshold-explicit}
			\beta>\frac3{\theta_*}-1,
		\end{equation}
		then \(\theta\) may be chosen sufficiently close to \(\theta_*\) that
		\eqref{eq:beta-condition-final} holds.  Summing
		\eqref{eq:interp-j-simplified} over \(j\ge2\) and invoking Step~1 gives a
		convergent operator series.  By Lemma~\ref{lem:dyadic-reconstruction}, this
		series represents the radial principal-value operator \(T_\Omega\). As the
		admissible exponents approach the formal
		symmetric boundary \(p_1=p_2=2\) from within the region \(p>1\), one has
		\(\theta_*\uparrow1\), and the threshold tends to \(\beta>2\). The boundary
		itself corresponds to \(p=1\) and is not covered by the theorem.
	\end{proof}
	\section{Comparison of the kernel classes}
	
	We give two explicit counterexamples showing that, for every \(\alpha,\beta>0\),
	the Orlicz class \(L(\log L)^\alpha(\Sph)\) and the directional class
	\(\cK_{1/2,\beta}(\Sph)\) are incomparable. The first establishes the
	non-inclusion of the Orlicz class in the directional class; the second gives a
	self-contained construction for the reverse non-inclusion.
	
	For the first non-inclusion, we concentrate the kernel near a direction where
	the directional weight becomes singular.  The singularity is chosen mild
	enough to preserve \(L(\log L)^\alpha\)-integrability, but strong enough that
	the corresponding directional integral diverges.  Subtracting the mean then
	gives the required cancellation without changing either property.
	
	\begin{proposition}\label{prop:Orlicz-not-Klog}
		Let \(\alpha>0\) and \(\beta>0\). Then
		\[
		L(\log L)^\alpha(\Sph)\not\subset \cK_{1/2,\beta}(\Sph).
		\]
	\end{proposition}
	
	\begin{proof}
		Fix \(\xi_0=(1,0)\in\Sph\) and define the nonnegative Borel function
		\begin{equation}\label{eq:Orlicz-not-Klog-example}
			\Omega_0(\theta)
			:=
			\begin{cases}
				\displaystyle
				\frac{1}
				{|\theta\cdot\xi_0|^{1/2}
					\bigl(\log \frac1{|\theta\cdot\xi_0|}\bigr)^{\beta+1}},
				&0<|\theta\cdot\xi_0|<e^{-2},\\[6pt]
				0,&\text{otherwise}.
			\end{cases}
		\end{equation}
		Thus the function is defined to be zero on the set
		\(\{\theta:\theta\cdot\xi_0=0\}\).
		The local model
		\(t^{-1/2}(\log(1/t))^{-\beta-1}\) is integrable at zero, so
		\(\Omega_0\in L^1(\Sph)\). Set
		\[
		\mu_0:=\int_{\Sph}\Omega_0\,d\sigma,
		\qquad
		\Omega:=\Omega_0-\mu_0.
		\]
		Then \(\Omega\) is Borel measurable, antipodally even, and has mean zero. We
		claim that
		\[
		\Omega\in L(\log L)^\alpha(\Sph)
		\qquad\text{but}\qquad
		\Omega\notin \cK_{1/2,\beta}(\Sph).
		\]
		
		We first verify that \(\Omega\notin \cK_{1/2,\beta}(\Sph)\).
		Since
		\[
		\frac1{t^{1/2}(\log(1/t))^{\beta+1}}\longrightarrow\infty
		\qquad (t\downarrow0),
		\]
		there exists \(0<\varepsilon_0<e^{-2}\) such that
		\(\Omega_0(\theta)\ge2\mu_0\) whenever
		\(0<|\theta\cdot\xi_0|<\varepsilon_0\). On this set,
		\(|\Omega(\theta)|\ge\Omega_0(\theta)/2\). Testing the definition of
		\(\cK_{1/2,\beta}\) at \(\xi=\xi_0\) therefore gives
		\begin{equation}\label{eq:Klog-lower}
			\int_{\Sph}
			\frac{|\Omega(\theta)|}{|\theta\cdot\xi_0|^{1/2}}
			\Bigl(\log\frac1{|\theta\cdot\xi_0|}\Bigr)^\beta
			\,d\sigma(\theta)
			\ge
			\frac12\int_{\{0<|\theta\cdot\xi_0|<\varepsilon_0\}}
			\frac{d\sigma(\theta)}{|\theta\cdot\xi_0|
				\log\frac1{|\theta\cdot\xi_0|}}.
		\end{equation}
		
		To analyze the integral on the right, parametrize \(\Sph\) by
		\(\theta(\varphi)=(\cos\varphi,\sin\varphi)\), \(\varphi\in[0,2\pi)\).
		Under our normalization, \(d\sigma(\theta)=d\varphi/(2\pi)\), and
		\[
		|\theta(\varphi)\cdot\xi_0|=|\cos\varphi|.
		\]
		Near \(\varphi=\frac\pi2\) one has \(|\cos\varphi|\sim|\varphi-\frac\pi2|\), and
		the same holds near \(\varphi=\frac{3\pi}{2}\). Therefore there exist constants
		\(c_1,c_2>0\) such that
		\[
		c_1|\varphi-\tfrac\pi2|\le |\cos\varphi|\le c_2|\varphi-\tfrac\pi2|
		\]
		whenever \(|\varphi-\frac\pi2|\) is sufficiently small, and similarly near
		\(\frac{3\pi}{2}\). It follows from \eqref{eq:Klog-lower} that
		\[
		\int_{\Sph}
		\frac{|\Omega(\theta)|}{|\theta\cdot\xi_0|^{1/2}}
		\Bigl(\log\frac1{|\theta\cdot\xi_0|}\Bigr)^\beta
		\,d\sigma(\theta)
		\gtrsim
		\int_0^\varepsilon \frac{dt}{t\log(1/t)}
		=+\infty
		\]
		for some sufficiently small \(\varepsilon>0\). Thus
		\(\Omega\notin \cK_{1/2,\beta}(\Sph)\).
		
		It remains to prove that \(\Omega\in L(\log L)^\alpha(\Sph)\).
		We begin with \(\Omega_0\). Since its support is contained in
		\(\{|\theta\cdot\xi_0|<e^{-2}\}\), it suffices to check integrability near the
		singular set \(\{\theta:\theta\cdot\xi_0=0\}\). Using the same parametrization
		and the equivalence \(|\cos\varphi|\sim |\varphi-\frac\pi2|\) near
		\(\frac\pi2\), as well as its analogue near \(\frac{3\pi}{2}\), it suffices to
		prove that
		\begin{equation}\label{eq:Orlicz-reduced}
			\int_0^\varepsilon
			\frac{1}{t^{1/2}(\log(1/t))^{\beta+1}}
			\Biggl(
			\log\Bigl(
			e+\frac{1}{t^{1/2}(\log(1/t))^{\beta+1}}
			\Bigr)
			\Biggr)^\alpha
			\,dt
			<\infty
		\end{equation}
		for some sufficiently small \(\varepsilon>0\).
		
		For \(0<t<e^{-2}\) we have \(\log(1/t)\ge 2\), and therefore
		\[
		\frac{1}{t^{1/2}(\log(1/t))^{\beta+1}}
		\le t^{-1/2}.
		\]
		Since \(t^{-1/2}\ge 1\) for \(0<t<1\), we obtain
		\[
		\log\Bigl(
		e+\frac{1}{t^{1/2}(\log(1/t))^{\beta+1}}
		\Bigr)
		\le
		\log(e+t^{-1/2})
		\lesssim \log(1/t).
		\]
		Consequently, the integrand in \eqref{eq:Orlicz-reduced} is bounded by
		\[
		\frac{(\log(1/t))^\alpha}{t^{1/2}(\log(1/t))^{\beta+1}}
		=
		t^{-1/2}(\log(1/t))^{\alpha-\beta-1}.
		\]
		Since
		\[
		\int_0^\varepsilon t^{-1/2}(\log(1/t))^{\alpha-\beta-1}\,dt<\infty,
		\]
		it follows that \eqref{eq:Orlicz-reduced} holds. Hence
		\(\Omega_0\in L(\log L)^\alpha(\Sph)\). Since \(\sigma(\Sph)<\infty\),
		constants belong to \(L(\log L)^\alpha(\Sph)\); moreover, this Orlicz space is
		a vector space. Hence
		\(\Omega=\Omega_0-\mu_0\in L(\log L)^\alpha(\Sph)\).
		
		Thus \(\Omega\) has mean zero, belongs to
		\(L(\log L)^\alpha(\Sph)\), and does not belong to
		\(\cK_{1/2,\beta}(\Sph)\).
	\end{proof}
	
	The reverse non-inclusion requires a different construction.  A single
	concentrated spike would be detected by some direction in the supremum defining
	\(\cK_{1/2,\beta}\), so we distribute many small spikes almost uniformly around
	the circle.  Their locations keep every directional integral bounded, while
	their heights and total masses are chosen so that the Orlicz modular diverges.
	After subtracting the mean, the resulting kernel remains in the directional
	class but lies outside \(L(\log L)^\alpha\).
	
	\begin{proposition}\label{prop:Klog-not-Orlicz}
		Let \(\alpha>0\) and \(\beta>0\). Then
		\[
		\cK_{1/2,\beta}(\Sph)\not\subset
		L(\log L)^\alpha(\Sph).
		\]
	\end{proposition}
	
	\begin{proof}
		Fix \(a\) with \(\frac12<a<1\).  We first construct a nonnegative function in
		\(\cK_a(\Sph)\) that does not belong to \(L(\log L)^\alpha(\Sph)\), and then
		subtract its mean. It will then suffice to use the embedding
		\(\cK_a(\Sph)\subset \cK_{1/2,\beta}(\Sph)\).
		
		Identify \(\Sph\) with \([0,2\pi)\) equipped with its normalized angular
		measure.  For each integer \(k\ge1\), set
		\[
		h_k:=4^k,\qquad w_k:=4^{-k}k^{-\alpha-1},
		\]
		and choose
		\[
		N_k:=\left\lceil
		w_k\,4^{\frac{k}{1-a}}k^{\frac{2}{1-a}}
		\right\rceil.
		\]
		Partition \(\Sph\) into \(N_k\) arcs of equal normalized measure \(1/N_k\). At
		the center of each arc, choose a closed subinterval \(I_{k,j}\) of normalized
		measure
		\[
		\sigma(I_{k,j})=\delta_k:=\frac{w_k}{N_k},
		\qquad 1\le j\le N_k.
		\]
		Let
		\[
		E_k:=\bigcup_{j=1}^{N_k}I_{k,j},
		\qquad
		\Omega_0(\theta):=\sum_{k=1}^\infty h_k\one_{E_k}(\theta).
		\]
		Each \(E_k\) is Borel, and hence the nonnegative function \(\Omega_0\) is
		Borel measurable. Moreover, \(\sigma(E_k)=w_k\) and, by Tonelli's theorem,
		\[
		\int_{\Sph}\Omega_0(\theta)\,d\sigma(\theta)
		=\sum_{k=1}^\infty h_k\sigma(E_k)
		=\sum_{k=1}^\infty k^{-\alpha-1}<\infty.
		\]
		Thus \(\Omega_0\) is finite almost everywhere and belongs to \(L^1(\Sph)\).
		We redefine it to be zero on the Borel null set where the displayed series
		is infinite, and set
		\[
		\mu_0:=\int_{\Sph}\Omega_0\,d\sigma,
		\qquad
		\Omega:=\Omega_0-\mu_0.
		\]
		Then \(\Omega\) is Borel measurable and has mean zero.
		
		We first show that \(\Omega_0\notin L(\log L)^\alpha(\Sph)\). Put
		\[
		A_k:=E_k\setminus\bigcup_{m>k}E_m.
		\]
		The sets \(A_k\) are pairwise disjoint, and
		\[
		\sigma(A_k)\ge \sigma(E_k)-\sum_{m>k}\sigma(E_m)
		=w_k-\sum_{m>k}w_m.
		\]
		Since \(w_m=4^{-m}m^{-\alpha-1}\),
		\[
		\sum_{m>k}w_m
		\le (k+1)^{-\alpha-1}\sum_{m=k+1}^\infty 4^{-m}
		=\frac13 (k+1)^{-\alpha-1}4^{-k}
		\le \frac13 w_k.
		\]
		Hence \(\sigma(A_k)\ge \frac23w_k\). Almost everywhere on \(A_k\) one has
		\(\Omega_0(\theta)\ge h_k\), and therefore
		\[
		\begin{aligned}
			\int_{\Sph}\Omega_0(\theta)
			\bigl(\log(e+\Omega_0(\theta))\bigr)^\alpha\,d\sigma(\theta)
			&\ge
			\sum_{k=1}^\infty h_k\bigl(\log(e+h_k)\bigr)^\alpha\sigma(A_k)\\
			&\gtrsim
			\sum_{k=1}^\infty 4^k k^\alpha\,4^{-k}k^{-\alpha-1}
			=\sum_{k=1}^\infty \frac1k
			=\infty.
		\end{aligned}
		\]
		Moreover, the divergence persists at every Luxemburg scale. For each
		\(\lambda>0\), choose \(k_\lambda\) so that, for every
		\(k\ge k_\lambda\),
		\(\log(e+h_k/\lambda)\gtrsim_\lambda k\), and hence
		\[
		\int_{\Sph}\Phi_\alpha\!\left(\frac{\Omega_0}{\lambda}\right)d\sigma
		\ge
		\sum_{k\ge k_\lambda}
		\frac{h_k}{\lambda}
		\bigl(\log(e+h_k/\lambda)\bigr)^\alpha\sigma(A_k)
		=\infty.
		\]
		Thus \(\Omega_0\notin L(\log L)^\alpha(\Sph)\). Since constants belong to
		\(L(\log L)^\alpha(\Sph)\) and this Orlicz space is a vector space, the
		identity \(\Omega_0=\Omega+\mu_0\) implies that
		\(\Omega\notin L(\log L)^\alpha(\Sph)\) as well.
		
		It remains to verify the fractional directional condition. For fixed
		\(\xi\in\Sph\), define
		\[
		I_k(\xi):=\int_{E_k}\frac{d\sigma(\theta)}{|\theta\cdot\xi|^a}.
		\]
		Let \(P_\xi=\{\xi^\perp,-\xi^\perp\}\) denote the two zeros of
		\(\theta\mapsto\theta\cdot\xi\) on \(\Sph\). Call \(I_{k,j}\) bad if its
		ambient partition arc meets \(P_\xi\). Since each point of \(P_\xi\) belongs to
		at most two adjacent partition arcs, there are at most four bad intervals.
		Their total contribution is bounded by
		\[
		\sum_{I_{k,j}\,{\rm bad}}
		\int_{I_{k,j}}|\theta\cdot\xi|^{-a}\,d\sigma(\theta)
		\lesssim_a \delta_k^{1-a},
		\]
		since, near either zero, \(|\theta\cdot\xi|\) is comparable to the angular
		distance to that zero.
		
		Group the remaining good intervals according to their grid distance \(m\) from
		\(P_\xi\). For each fixed \(m\), the number of such intervals is uniformly
		bounded, and on each of them
		\(|\theta\cdot\xi|\gtrsim m/N_k\). Hence
		\[
		\sum_{I_{k,j}\,{\rm good}}
		\int_{I_{k,j}}|\theta\cdot\xi|^{-a}\,d\sigma(\theta)
		\lesssim
		\delta_k N_k^a
		\sum_{m=1}^{N_k}m^{-a}
		\lesssim_a \delta_kN_k=w_k,
		\]
		since \(0<a<1\). Consequently
		\[
		\sup_{\xi\in\Sph} I_k(\xi)
		\lesssim_a \delta_k^{1-a}+w_k.
		\]
		Using \(N_k\ge w_k4^{k/(1-a)}k^{2/(1-a)}\), we have
		\[
		\delta_k=\frac{w_k}{N_k}
		\le 4^{-\frac{k}{1-a}}k^{-\frac{2}{1-a}},
		\qquad
		\delta_k^{1-a}\le 4^{-k}k^{-2}.
		\]
		Therefore
		\[
		\begin{aligned}
			\sup_{\xi\in\Sph}
			\int_{\Sph}\frac{\Omega_0(\theta)}{|\theta\cdot\xi|^a}
			\,d\sigma(\theta)
			&\le
			\sum_{k=1}^\infty h_k\sup_{\xi\in\Sph}I_k(\xi)\\
			&\lesssim_a
			\sum_{k=1}^\infty 4^k(4^{-k}k^{-2}+4^{-k}k^{-\alpha-1})
			<\infty.
		\end{aligned}
		\]
		Thus \(\Omega_0\in\cK_a(\Sph)\).  Moreover, since \(a<1\), rotational invariance
		and local integrability at the two zeros give
		\[
		C_a:=\sup_{\xi\in\Sph}
		\int_{\Sph}|\theta\cdot\xi|^{-a}\,d\sigma(\theta)<\infty.
		\]
		Consequently,
		\[
		\norm{\Omega}_{\cK_a}
		\le \norm{\Omega_0}_{\cK_a}+\mu_0 C_a<\infty,
		\]
		so the centered function \(\Omega\) also lies in \(\cK_a(\Sph)\).
		
		Finally, since \(a>\frac12\), for \(0<t\le1\) one has
		\[
		t^{-1/2}\left(\log\frac1t\right)^\beta\le C_{a,\beta}t^{-a}.
		\]
		Applying this with \(t=|\theta\cdot\xi|\) gives
		\[
		\norm{\Omega}_{\cK_{1/2,\beta}}
		\le C_{a,\beta}\norm{\Omega}_{\cK_a}<\infty.
		\]
		Therefore \(\Omega\) has mean zero, belongs to
		\(\cK_{1/2,\beta}(\Sph)\), and does not belong to
		\(L(\log L)^\alpha(\Sph)\).
	\end{proof}
	
	\begin{proof}[Proof of Theorem~\ref{thm:incomparability}]
		The two non-inclusions follow from
		Propositions~\ref{prop:Orlicz-not-Klog}
		and~\ref{prop:Klog-not-Orlicz}.  Taking \(\alpha=1\) gives the final
		assertion concerning \(L\log L(\Sph)\).
	\end{proof}

	\end{document}